\documentclass[11pt, reqno]{amsart}
\usepackage{amsfonts,amssymb,amscd,amstext,mathrsfs,color,comment}
\usepackage[dvipsnames]{xcolor}
\usepackage{graphicx}
\usepackage{tikz}
\usepackage{hyperref}
\usepackage{mathtools}

\numberwithin{equation}{section}

\newtheorem{theorem}{Theorem}[section]
\newtheorem{corollary}[theorem]{Corollary}
\newtheorem{lemma}[theorem]{Lemma}
\newtheorem{proposition}[theorem]{Proposition}

\theoremstyle{definition}
\newtheorem{remark}[theorem]{Remark}

\newtheorem{definition}[theorem]{Definition}
\newtheorem{example}[theorem]{Example}

\newcommand{\C}{\mathbb{C}}
\newcommand{\D}{\mathbb{D}}
\newcommand{\B}{\mathbb{B}}

\newcommand{\R}{\mathbb{R}}

\newcommand{\id}{\mathrm{id}}
\renewcommand{\H}{\mathbb{H}}
\renewcommand{\Im}{{\operatorname{Im}\,}}
\renewcommand{\Re}{{\operatorname{Re}\,}}

\newcommand{\psh}{\operatorname{Psh}}

\begin{document}

\title{The pluricomplex Poisson kernel for convex \\ finite type domains}

\author{Leandro Arosio, Filippo Bracci,  Matteo Fiacchi}

\address{L. Arosio, M. Fiacchi: Dipartimento di Matematica, Universit\`a degli Studi di Roma ``Tor Vergata'', Via della Ricerca Scientifica 1, 00133, Roma, Italy}\email{arosio@mat.uniroma2.it, fiacchi@mat.uniroma2.it}

\address{F. Bracci: Dipartimento di Matematica e Informatica ``Ulisse Dini'', Universit\`a di Firenze, Viale Morgagni 67/A, Firenze, Italy}\email{filippo.bracci@unifi.it}

\thanks{Partially supported by the MIUR Excellence Department Project 2023-2027 MatMod@Tov awarded to the	Department of Mathematics, University of Rome Tor Vergata, by PRIN Real and Complex Manifolds: Topology, Geometry and holomorphic dynamics n.2017JZ2SW5 and by GNSAGA of INdAM}


\begin{abstract}
Given a bounded convex domain $D\subset \C^n$ of finite D'Angelo type and a boundary point $\xi\in \partial D$, we prove that the homogeneous complex Monge-Amp\`ere equation $(dd^cu)^n=0$ possesses a continuous strictly negative solution $\Omega_\xi$ that vanishes on $\partial D\setminus \{\xi\}$ and has a simple pole at $\xi$.
We establish that $\Omega_\xi(z)$ equals (up to  sign) the normal derivative at $\xi$ of the pluricomplex Green function $G_z$, and its sublevel sets are the horospheres centered at $\xi$. Moreover, $\Omega_\xi$  is the maximal element of the family of psh functions with a prescribed behavior at $\xi$, that is, it satisfies a Phragmen-Lindel\"of type-theorem and provides a reproducing formula for plurisubharmonic functions. Consequently, $\Omega_\xi$ serves as a generalisation of the  classical Poisson kernel of the unit disc.
Our approach, based on metric methods and scaling techniques, allows our results to be applied to strongly convex domains with $C^2$-smooth boundaries as well. In the course of the proof, we also establish a novel estimate of the Kobayashi distance near boundary points. 
\end{abstract}

\maketitle
\tableofcontents

\section{Introduction}
In their celebrated work \cite{BedTay}, Bedford and Taylor proposed using the complex Monge-Amp\`ere operator as the foundation for a pluripotential theory in $\C^n$. Recall that, if $u\colon D\to \R$ is a $C^2$-smooth map on a domain $D\subset \C^n$, the complex Monge-Amp\`ere operator $u\mapsto (dd^cu)^n$ is defined by
$$(dd^cu)^n=\underbrace{dd^cu\wedge\cdots\wedge dd^cu}_{n}=4^nn!\det\left(\frac{\partial^2u}{\partial z_j\partial\bar{z}_k}\right)dV,$$
where $d^c$ denotes the operator $i(\bar\partial-\partial)$ and  $dV$  the standard volume form. 
To develop a potential theory in $\mathbb{C}^n$, one must define the complex Monge-Amp\`ere operator on plurisubharmonic functions that are not necessarily $C^2$-smooth. Bedford and Taylor showed that the complex Monge-Amp\`ere operator $(dd^c)^n$ can be defined for every plurisubharmonic function $u$ in $L^\infty_{loc}(D)$, and that $(dd^c)^n(u)=0$ on $D$ is equivalent to $u$ being maximal in the sense of Sadullaev (see Section \ref{preliminaries}).

\subsection{The pluricomplex Green function}
A striking instance of this approach is the introduction by Lempert \cite{Lemp1} of the pluricomplex Green function on a bounded strongly convex domain $D\subset \C^n$ with $C^{r,\alpha}$ boundary, $r\geq 2$, $0<\alpha<1$. Indeed, Lempert  proved that for all $w\in D$ 
there exists a  continuous   function $G_{w}\colon \overline D\to [-\infty,0]$, $C^{r-1,\alpha-\varepsilon}$-smooth on $\overline D\setminus\{w\}$ where $0<\varepsilon<<1$, solving the following Monge--Amp\`ere equation:
\begin{equation}\label{greenintro}
\begin{cases}
u\in {\rm Psh}(D)\cap L^\infty_{\rm{loc}}( D\setminus \{w\}),\\
(dd^cu)^n=0 \quad\textrm{on} \quad D\setminus \{w\},\\
u(z)\stackrel{z\to\eta}\longrightarrow 0 \ \mbox{ if }\eta\in\partial D,\\
u - \log|z-w| =O(1)\quad\textrm{as} \quad z\to w.
\end{cases}
\end{equation}
The function $G_w$ is constructed using the theory of complex geodesic introduced in  \cite{Lemp1} to define  a homeomorphism
$\Phi_w$ from  $\overline D$ to the closed unit ball $ \overline \B^n\subset \C^n$, called {\it  spherical representation}, sending $w$ to the origin and sending holomorphically every complex geodesic in $D$ through $w$ in a complex geodesic in $\B^n$  through $0$. The map $\Phi_w$  is moreover $C^{r-1,\alpha-\varepsilon}$-smooth on $\overline{D}\setminus\{w\}$.
The pluricomplex Green function $G_w$ is then defined as the pull-back via $\Phi_w$ of the Green function of the ball, that is $$G_w(z):=\log|\Phi_w(z)|=\log\tanh(k_D(z,w)/2),$$ where $k_D$ denotes the Kobayashi distance of $D$. 
This formula shows that in strongly convex domains  pluripotential theory and intrinsic metrics are strongly tied: as an example the sublevel sets of $G_w$ are the Kobayashi balls centered at $w$. Another important consequence is that the pluricomplex Green function is symmetric: $G_w(z)=G_z(w).$
Notice moreover that by construction $G_w$ is harmonic on every complex geodesic through 0. 

Unfortunately, Lempert's methods are not available for more general domains. However, Demailly \cite{Demailly} (see also Klimek \cite{Kli85,KlimBook}) proved that on a bounded hyperconvex domain $D\subset \C^n$, the Monge-Amp\`ere equation \eqref{greenintro} has a unique continuous solution $G_w\colon \overline D\to [-\infty,0]$. This solution is defined via the Perron's ``balayage'' method as the upper envelope of all plurisubharmonic functions $u\leq 0$ on $D$ such that $u(z)-\log\|z-w\|\leq O(1)$ as $z\to w$.
The function $G_w(z)$ is not  symmetric in general, indeed Bedford and Demailly \cite{BedDem} provided as a counterexample a bounded strongly pseudoconvex domain of $\mathbb{C}^2$ with a real analytic boundary. 

In case $D$ is a bounded convex domain, $G_w(z)$ is symmetric. This is because, if $D$ is a bounded convex domain, the Carath\'eodory distance $c_D$ equals the Kobayashi distance $k_D$  by a result of Lempert \cite{Lem82} (also see \cite{RW83}). Therefore, the function $z\mapsto \log\tanh(k_D(z,w)/2)$ is continuous (as function with values in $[-\infty, 0)$) and plurisubharmonic on $D$. Since it is harmonic on every complex geodesic passing through $w$, it is maximal. The uniqueness of the pluricomplex Green function implies that $G_w(z)= \log\tanh(k_D(z,w)/2)$, making it symmetric.

Błocki \cite{Blocki} proved that $G_w\in C^{1,1}(\overline{D}\setminus\{w\})$ if $D$ is a strongly pseudoconvex, $C^\infty$-smooth domain. However, Bedford and Demailly \cite{BedDem} provided an example of a bounded, strongly pseudoconvex domain in $\C^2$ with a real analytic boundary such that $G_w$ is not $C^2$ on $\overline{D}\setminus\{w\}$.

\subsection{Demailly's pluriharmonic measure} If $D$ is a hyperconvex bounded domain in $\C^n$, Demailly \cite{Demailly} proved that for every $z\in D$ there exists a probability  measure $\mu_z$ on $\partial D$, called the {\it pluriharmonic  measure} at $z$ (also called {\sl Poisson measure} in \cite{BST}), such that for every plurisubharmonic function $F$ on $D$, continuous on $\overline{D}$,
\[
F(z)=\int_{\sigma\in \partial D} F(\sigma)d\mu_z(\sigma)-\frac{1}{(2\pi)^n}\int_{w\in D}|G_z(w)|dd^c F(w)\wedge (dd^c G_z)^{n-1}(w).
\]
In case $n=1$, $\mu_z$ is just the harmonic measure of the domain $D$. In particular, if $D$ is the unit disc of $\C$, $d\mu_z(\sigma)=|P(z,\sigma)|\frac{d\sigma}{2\pi}$, where $$P(z,\sigma)=\Re \left(\frac{z+\sigma }{z-\sigma}\right)=-\frac{1-\|z\|^2}{|1-z\bar\sigma|^2}$$ is the classical (negative) Poisson kernel. Consequently, in the case of the unit disc, Demailly's formula reduces to the classical Riesz-Poisson-Green's formula (see, {\sl e.g.}, \cite{Ran}).  

\subsection{The pluricomplex Poisson kernel}
The second named author, together with Patrizio and Trapani \cite{BP,BPT}, introduced and studied the  pluricomplex Poisson kernel on  bounded  strongly convex domains $D\subset \C^n$ with $C^\infty$-smooth boundary. It is proved in \cite{BP}  that  for all $\xi\in \partial D$  there exists a $C^\infty$-smooth  function $\Omega_\xi$ on $D$ solving the following Monge--Amp\`ere equation:
\begin{equation}\label{poissonintro}\begin{cases*}
u\in {\rm Psh}(D)\cap L^\infty_{\rm{loc}}( D), \\
 u<0  \mbox{ on }D,\\
(dd^cu)^n=0 \quad\textrm{on} \quad D,\\
				u(z)\stackrel{z\to\eta}\longrightarrow 0 \ \mbox{ if }\eta\in\partial D\backslash\{\xi\}\\
				u(z)\approx-\|z-\xi\|^{-1} \mbox{ as }z\to\xi\mbox{ non-tangentially}.
			\end{cases*}
		\end{equation}
The pluricomplex Poisson kernel  $\Omega_\xi$  is constructed using a generalization to the boundary of Lempert's theory introduced by Chang--Hu--Lee \cite{CHL} under the assumption that the boundary is $C^{r}$-smooth with $r\geq 14$. Chang--Hu--Lee define a 
homeomorphism $\Phi_\xi\colon \overline D\to \overline \B^n$, called the {\it boundary spherical representation}, which sends $\xi$ to $e_1$ and sends holomorphically every complex geodesic with endpoint $\xi$ in $D$ to a complex geodesic with endpoint $e_1$ in $\B^n$. Moreover the map $\Phi_\xi$ is  $C^{r-5}$-smooth on $\overline{D}\setminus\{\xi\}$.
The function $\Omega_\xi$ is defined in \cite{BP} as the  pull-back via $\Phi_\xi$ of the function $\Omega_{e_1}^{\B^n}(z)=-\frac{1-\|z\|^2}{|1-z_1|^2}.$
By construction $\Omega_\xi$ is $C^\infty$-smooth and is harmonic on complex geodesics with endpoint $\xi$.
Moreover, it is proved in  \cite{BP,BPT} that  $\Omega_\xi$ satisfies several  notable properties establishing it as a natural generalization of the classical Poisson kernel in the disc, which we now discuss.

{ The  Poisson kernel with pole at $1$ in the unit disc is the maximal element of the family of all negative subharmonic functions in the unit disc which have at least a non-tangential pole at $1$ (see, {\sl e.g.}, \cite[Lemma~5.2]{BPT}). Translating such a result on the upper-half plane, one sees that this maximality property is equivalent to the classical Phragmen-Lindel\"of theorem. This is the reason such a result  is also called  a ``Phragmen-Lindel\"of theorem''}. The function $\Omega_\xi$ satisfies a generalised Phragmen-Lindel\"of theorem: it is the greatest element of the family $\mathscr{G}_\xi$ of all strictly negative plurisubharmonic functions $u$ on $D$ that satisfy
\begin{equation}\label{MAmaxintro}
		\limsup_{t\to1^-}u(\gamma(t))(1-t)\leq-2\Re\frac{1}{\langle \gamma'(1),n_\xi\rangle}, 
\end{equation}
for all  $C^1$ curves $\gamma\colon [0,1]\to D$ such that $\gamma(1)=\xi$ and $\gamma'(1)$ is not contained in  $ T_\xi\partial D$.   By $\langle\cdot,\cdot\rangle$ we denote the standard Hermitian  product on $\C^n$.

A remarkable instance of a strong potential-metric tie is given by the following formula (which we will refer to as the {\it Poisson-horofunction formula}):

\begin{equation}\label{poisson-horofunction-intro}
h_{\xi,p}(z)=\log |\Omega_\xi(p)|-\log  |\Omega_\xi(z)|,
\end{equation}
where $h_{\xi,p}\colon D\to \R$ is the {\it horofunction} defined by
$h_{\xi,p}(z):=\lim_{w\to z}k_D(z,w)-k_{D}(w,p).$
In particular, it follows from \eqref{poisson-horofunction-intro} that the sublevel sets of $\Omega_\xi$ coincide with the sublevel sets of the horofunction, which are called {\it horospheres} centered at $\xi$. The fact that the limit in the definition of the horofunction $h_{\xi,p}$ exists is non-trivial and was proved by Abate and Venturini (see \cite[Theorem (2.6.47)]{Abatebook}),  assuming $C^3$ boundary, using 
the $C^1$-regularity of the Green function $G_w$, together with the relation $G(z,w)=\log\tanh(k_D(z,w)/2)$. Indeed they prove the following formula:
\begin{equation}\label{abate-venturini-intro}
h_{\xi,p}(z)=\lim_{w\to z}k_D(z,w)-k_{D}(w,p)=\log\left(\frac{\partial G_{p}}{\partial n_\xi}(\xi)\right)-\log\left(\frac{\partial G_{z}}{\partial n_\xi}(\xi)\right),
\end{equation}
where $n_\xi$ is the unit outward normal vector to $\partial D$ in $\xi$. 

Putting together formulas \eqref{abate-venturini-intro} and \eqref{poisson-horofunction-intro} one obtains that $\Omega_\xi(z)$  coincides (up to  sign) with  the normal derivative in $\xi$ of the Green function $G_z$, that is
\begin{equation}\label{derivata-green-intro}
\Omega_\xi(z)=-\frac{\partial G_z}{\partial n_\xi}(\xi).
\end{equation}

Finally, if  $\mu_z$ denotes the  pluriharmonic  measure, then  
\begin{equation}\label{demailly-intro}
d\mu_z(\xi)= \frac{1}{(2\pi)^n}|\Omega_\xi(z)|^n \omega_{\partial D}(\xi),
\end{equation}
  where $\omega_{\partial D}$ is a natural volume form on $\partial D$ (see Section~\ref{sectDem} for precise definitions). 


 Huang and Wang \cite{HuWa} proved (for the more general case of strongly linearly convex domains) that the method of pulling back $\Omega_{\B^n}$ via a boundary spherical representation can be improved to obtain a continuous solution to the Monge-Amp\`ere equation \eqref{poissonintro},  still retaining the notable properties mentioned above, assuming only a $C^3$-smooth boundary.

Recently, Saracco, Trapani, and the second author \cite{BST} constructed  continuous solutions $\Omega_\xi$ to the Monge-Amp\`ere equation \eqref{poissonintro} on  bounded strongly pseudoconvex domains with $C^\infty$-smooth boundaries, defining $\Omega_\xi$ as the upper envelope of the family $\mathscr{G}_\xi$  in \eqref{MAmaxintro}. 
The validity of the Poisson-horofunction formula \eqref{poisson-horofunction-intro} was left open in \cite{BST}. In this paper, we show in Section~\ref{controesempioanello} that, in fact, the Poisson-horofunction formula does not hold in general for bounded strongly pseudoconvex domains, since  horospheres do not coincide in general with the sublevel sets of $\Omega_\xi$. 

\subsection{Convex domains of finite type}
In this paper we  study  the pluricomplex Poisson kernel  on  bounded convex domains $D\subset \C^n$ of  finite D'Angelo type. 
If a domain is convex with a $C^\infty$-smooth boundary, McNeal \cite{mcneal}(see also \cite{BS}) showed that the D'Angelo finite type condition is equivalent to the finite line-type condition, {\sl i.e.}, the order of contact of the boundary with complex affine lines is bounded from above on  $\partial D$.  
 The boundedness assumption is not essential in most of our results, as we  only need to assume that $D$ does not contain any complex line. Moreover we need to assume the finite type condition only locally around a point. For the sake of clarity we will ignore this technical issue in these paragraphs.

{ Our approach is completely new and is based on  the results on strong asymptoticity of curves obtained in \cite{Arosio-Fiacchi} by the first and third authors (see also \cite{AFGG}).}
We say that two curves in $D$ with endpoint $\xi\in\partial D$ are {\it strongly asymptotic} if, suitably reparametrized, their Kobayashi distance goes to 0 as they approach $\xi$. It is proved in \cite{Arosio-Fiacchi, AFGG} that complex geodesics with endpoint $\xi$ are strongly asymptotic, and as a consequence one can prove the existence of the horofunction $h_{\xi,p}$. Notice that Abate--Venturini's proof of \eqref{abate-venturini-intro} is not applicable since no regularity results are known for the Green function $G_w$ in the finite type case.

We now sketch the construction of the function $\Omega_\xi$. In contrast with what happens in the strongly convex case, a complex geodesic $\varphi$ with endpoint $\xi$ may approach $\xi$ (complex) tangentially, and the derivative $\varphi'(1)$ may not exist. However, it is proved in \cite{Arosio-Fiacchi, AFGG} that the non-tangential limit of the normal component of the derivative always exists and is a positive real number. Indeed, if $n_\xi$ denotes the outer normal vector at $\xi$, we have
$$\varphi'_N(1):=\angle \lim_{z\to 1}\langle \varphi'(z),n_\xi\rangle>0.$$

The pluricomplex Poisson kernel  $\Omega_\xi$ is then defined as $\Omega_\xi(z):=-\frac{1}{\varphi'_N(1)},$ where $\varphi$ is a complex geodesic with endpoint $\xi$ and such that $\varphi(0)=z$. While such a complex geodesic may not be uniquely determined, strong asymptoticity implies that $\varphi'_N(1)$ is well defined and is independent of the chosen geodesic.

The following is our main result.  It is worth noting that this construction   allows to define  the pluricomplex Poisson kernel for bounded strongly convex domains  with $C^2$-smooth boundaries.
\begin{theorem}\label{main}
Let $L\geq 2$. Let $D\subset \C^n$ be a bounded convex domain with $C^L$-smooth boundary of finite line type $L$.
Then $\Omega_\xi\colon D\to (-\infty,0)$ is a continuous  function solving   the following Monge--Amp\`ere equation:
\begin{equation*}\begin{cases*}
u\in {\rm Psh}(D)\cap L^\infty_{\rm{loc}}( D), \\
 u<0  \mbox{ on }D,\\
(dd^cu)^n=0 \quad\textrm{on} \quad D,\\
				u(z)\stackrel{z\to\eta}\longrightarrow 0 \ \mbox{ if }\eta\in\partial D\backslash\{\xi\}\\
				u(z)\approx -\|z-\xi\|^{-1} \mbox{ as }z\to\xi\mbox{ non-tangentially}.
			\end{cases*}
		\end{equation*}
Moreover, 
\begin{enumerate}
\item $\Omega_\xi$ is the greatest element of the family $\mathscr{F}_\xi$ defined below;
\item $h_{\xi,p}(z)=\log |\Omega_\xi(p)|-\log  |\Omega_\xi(z)|$;
\item $\Omega_\xi(z)=-\frac{\partial G_z}{\partial n_\xi}(\xi)$;
\item if $\mu_z$ denotes the pluriharmonic measure, then  $d\mu_z(\xi)=\frac{1}{(2\pi)^n} |\Omega_\xi(z)|^n \omega_{\partial D}(\xi)$, where $\omega_{\partial D}$ is a natural  form on $\partial D$ (see \eqref{Eq:Levi-form}).
\end{enumerate}
\end{theorem}

Here $\mathscr{F}_\xi$ is the family of all strictly negative plurisubharmonic functions on $D$ such that 
$$\limsup_{t\to1^-}u(\gamma(t))(1-t)\leq-2\Re\frac{1}{\gamma'_N(1)},$$ for all
$C^1$-smooth curves $\gamma\colon[0,1)\to D$ which $K'$-converge to $\xi$,   such that $$\gamma'_N(1):=\lim_{t\to1^-}\langle \gamma'(t),n_\xi\rangle$$ exists positive (the notion of $K'$-convergence will be discussed later on---in particular, non-tangential convergence implies $K'$-convergence). The difference in the definition of the family $\mathscr{F}_\xi$ and the previously introduced family $\mathscr{G}_\xi$  is due to the varying behaviour of complex geodesics.

The structure of the proof of Theorem \ref{main} is reversed compared to the proof in \cite{BP,BPT}. The Poisson-horofunction formula (2) is established using the strong asymptoticity of complex geodesics. This implies the continuity of the function $\Omega_\xi$. Since, by \cite{Arosio-Fiacchi,AFGG}, complex geodesics are strongly asymptotic to the inner normal segment, we obtain property (3). This property, in turn, implies that $\Omega_\xi$ is plurisubharmonic due to the symmetry of $G_w(z)$. All previous proofs of property (3) relied on the regularity of the pluricomplex Green function $G_w$ at $\xi$, which is not known in the convex finite type case. Instead, our approach provides a metric proof of the existence of the normal derivative of $G_w$ in $\xi$. Combining (2) and (3) yields Abate-Venturini's formula \eqref{abate-venturini-intro} in this setting.

The function $\Omega_\xi(z)$ goes to 0 as $z\to \eta\in\partial D\backslash\{\xi\}$ thanks to the Poisson-horofunction formula, together with  the Gromov hyperbolicity of $D$ which implies that a horosphere touches the boundary only at its center.
Again the Poisson-horofunction formula shows that $\Omega$ is harmonic on every complex geodesic with endpoint $\xi$, and thus $(dd^c \Omega_\xi)^n=0$.

Statement (4) is based on Demailly's theory  \cite{Demailly}, but since we do not know the regularity of $G_w$, we can not follow the same argument as in \cite{BPT}. Instead, we prove a surprising asymptotic estimate of the Kobayashi distance in terms of the pluricomplex Poisson kernel.  
Let $\delta_D(z):=\min\{\|z-\eta\|:\eta\in\partial D\}$ denote the Euclidean distance of $w$ from to the boundary $\partial D$. Recall that if $D\subset\C^n$ is a bounded convex domain with $C^2$-smooth boundary, then for all $w\in D$ there exist $C_1\leq C_2$ such that for all $z\in D$
 $$-\log\delta_D(z)+C_1\leq k_D(z,w)\leq -\log \delta_D(z)+C_2.$$
In this paper we prove the following very precise estimate:
 
  \begin{theorem}\label{main 2}
Let $L\geq 2$.  Let $D\subset \C^n$ be a bounded convex domain of finite line type $L$ with $C^L$-smooth boundary.
 Let $w\in D$, $\xi\in \partial D$, then	
	$$\lim_{z\to\xi}[k_D(z,w)+\log\delta_D(z)]=-\log\frac{|\Omega_\xi(w)|}{2}.$$
 \end{theorem}

\subsection{Acknowledgement}

We sincerely thank the anonymous referees for their helpful comments and suggestions which significantly improved the original manuscript.

\section{Preliminaries}\label{preliminaries}

If $D\subset \C^n$ is a domain, we denote by $\overline D^*$ ({\sl respectively} $\partial^*D$) the closure ({\sl resp.}, boundary) of $D$ in the one-point compactification of $\C^n$.

We use the following normalization (see e.g. Kobayashi \cite{kobayascione} and \cite{AFGG, Arosio-Fiacchi}) for the
	Poincar\'e   distance in the disc $\D\subset \C$:
	\begin{equation}\label{normalization}k_\D(0,z)=\log\frac{1+|z|}{1-|z|}= 2 \,{\rm arctanh}(|z|).
	\end{equation}
	Another common normalization of the Poincar\'e   distance in the disc is obtained  dividing \eqref{normalization} by 2. This results in a different normalization of the Kobayashi distance which is used {\sl e.g.} in \cite{BP, BPT, BST}.

We recall that a convex domain is said to be $\C$-proper if it does not contain any complex lines. This is equivalent to the Kobayashi completeness of the space (see \cite{Harris, Barth, Bracci-Saracco}).

\begin{definition}[Locally finite type]
		Let $D\subseteq \C^n$ be a   $\C$-proper convex domain.   Let $L\geq 2$,  and assume that $\partial D$ is of class $C^L$ in a neighborhood of a boundary point   $\xi\in\partial D$. Let
		$r$ be a defining function of class $C^L$ for $\partial D$ in a neighborhood of $\xi$.
		The point $\xi$ has {\it   type} $L$ if the maximum, as $v\in\C^n\setminus\{0\}$, of the order of vanishing of the function
	$\zeta \mapsto r(\xi+\zeta v)$ is $L$.		Notice that $L$ is necessarily an even number.
		 The convex domain $D$ has {\it finite  type} if there exists  $L\geq 2$ such that the boundary $\partial D$ is of class $C^L$ and every point has  type at most $L$.
		Finally, we say that a point $\xi\in \partial D$ has {\it locally  finite   type} if there exists   $L\geq 2$ and a neighborhood $U$ of $\xi$ such that $\partial D\cap U$ is of class $C^L$ and every point in $\partial D\cap U$ has line type at most $L$.
\end{definition}

The following definition is due to Sadullaev \cite{Sad}.
\begin{definition}[Maximal plurisubharmonic functions]
Let $D\subseteq\C^n$ be a domain. 
A plurisubharmonic function $u\colon D\to \R$ is called \textit{maximal} if for all open set $V\subset\subset D$ and for all plurisubharmonic function $v\colon V\to \R$  such that $\limsup_{z\to p}v(z)\leq u(p)$ for all $p\in\partial V$, it holds $v(z)\leq u(z)$ for all $z\in V$.
\end{definition}

By Bedford--Taylor the complex Monge-Amp\`ere operator $(dd^c)^n$ can be defined for every plurisubharmonic function $u$ which is in $L^\infty_{loc}(D)$.
\begin{theorem}[Bedford--Taylor \cite{BedTay}]
Let $D\subseteq\C^n$ be a domain. Let $u\in\mbox{Psh}(D)\cap L^\infty_{loc}(D)$.
 Then $u$ is maximal if and only if  $(dd^cu)^n=0$.
\end{theorem}

We will need the following sufficient condition for maximality. Let $\D\subset \C$ denote the unit disc.
\begin{proposition}\label{maximaldisk}\cite[Proposition 5.1.4]{BracciTrapani}
Let $D\subset\C^n$ be a domain and let $u\colon D\to\R$ be a plurisubharmonic function.  Assume that for all $z\in D$ there exists a proper holomorphic map $\varphi\colon \D\to D$ with $z\in\varphi(\D)$ such that $u\circ\varphi$ is harmonic on $\D$, then $u$ is maximal.
\end{proposition}

\section{Consequences of strong asymptoticity}
In this section we recall definitions and results from \cite{Arosio-Fiacchi} which will be needed in what follows. 
\subsection{Complex geodesics}
\begin{definition}[Complex geodesics]
Let $D\subset\C^n$ be a $\C$-proper convex domain. A \textit{complex geodesic} is a holomorphic map $\varphi\colon\D\to D$ such that
$$k_D(\varphi(\zeta_1),\varphi(\zeta_2))=k_\D(\zeta_1,\zeta_2)$$
for all $\zeta_1,\zeta_2\in \D$. 

\end{definition}
\begin{remark}\label{remarkgeodesic}
Let $D\subset\C^n$ be a $\C$-proper convex domain.
\begin{enumerate}
\item Every complex geodesic $\varphi:\D\to D$ admits a \textit{holomorphic left inverse} $\tilde\rho\colon D\to\D$ satisfying $\tilde\rho\circ\varphi=\id_\D$ (see e.g. \cite{JP}).
\item For all $z,w\in D$ there exists a complex geodesic $\varphi\colon\D\to D$ such that $z,w\in\varphi(\D)$. 
In particular, $c_D=k_D$, where $c_D$ is the Carathéodory distance.
\item Let  $\xi\in\partial D$ be a  point of locally finite  type. 
It is proved in \cite[Proposition 4.3]{Arosio-Fiacchi} that for any point $z\in D$ there exists a complex geodesic with $\varphi(0)=z$ and such that 
$\lim_{\R\ni t\to 1^-}\varphi(t)=\xi$. We say that  $\xi$ is the \textit{endpoint} of $\varphi$, and we denote $\xi=\varphi(1).$
 It is not known whether such complex geodesic is unique (uniqueness is known if $D$ is bounded and strongly convex with $C^{3}$-smooth boundary, see \cite[Theorem 2.6.45]{Abatebook}).
\end{enumerate}

\end{remark}

If $D\subset \C^n$ is a bounded strongly convex domain with $C^{r,\alpha}$-smooth boundary, where $r\geq 2$ and $\alpha\in (0,1)$, then (see e.g. \cite{HuWa}) any complex geodesics $\varphi\colon \D\to D$ extends $C^{r-1,\alpha}$-smooth to $\partial \D$.
By the classical Hopf's Lemma one has ${\rm Re}\langle\varphi'(1),n_\xi\rangle>0$ (see e.g. \cite[Remark 1.2]{HuWa}), and in particular the geodesic $\varphi$ approaches $\xi$ non-tangentially. Moreover since $\varphi$ is proper one has $i\varphi'(1)\in T_\xi\partial D$, concluding  that   $\langle\varphi'(1),n_\xi\rangle>0.$

Assume now that the boundary of $D$ is $C^\infty$-smooth.
Then it follows from  \cite[Eq.~(1.2)]{BPT} that 
\begin{equation}\label{ponte}\Omega_\xi(\varphi(0))=-\frac{1}{\langle\varphi'(1),n_\xi\rangle}.
\end{equation} The same result holds if the boundary is $C^3$-smooth thanks to \cite[Remark 4.4]{HuWa}.
 For a $\C$-proper convex domain $D\subset\C^n$  with a  point $\xi\in\partial D$  of locally finite  type the situation is quite different.
  
  	\begin{example}[Complex geodesics of the egg domain]\label{egggeodesics}
  	Let $m\geq 2$ be an even integer and define the {\it egg domain} 
  	\begin{equation}\label{eggdomain}
  		\mathbb{E}_m:=\{(z_0,z_1)\in\C^2:|z_0|^2+|z_1|^{m}<1\}\subset \C^2.
  	\end{equation}
  	The point $\xi:=(1,0)$ is a boundary point of type $m$.
  	For all $a\in\C$  the map $\varphi_a\colon \D\to \mathbb{E}_m$ given by 
  	\begin{equation}\label{formulageodetiche}
  		\varphi_a(\zeta):=\left(\frac{\zeta+|a|^{m}}{1+|a|^{m}},a\left(\frac{1-\zeta}{1+|a|^{m}}\right)^{2/m}\right)
  	\end{equation}
 is a  complex geodesic with endpoint $\xi$ (see \cite{Huang} and \cite{JPZ}). 
  	Notice that if $m>2$ and $a\neq0$ then $\lim_{t\to 1}\|\varphi_a'(t)\|=+\infty$ and $\varphi_a(t)$ converges to $\xi$ tangentially.
  \end{example}
  However, it is proved in \cite{Arosio-Fiacchi, AFGG} that the following non-tangential limit exists:
\begin{equation}\label{varphi'_N(1)}
\varphi'_N(1):=\angle \lim_{z\to 1}\langle \varphi'(z),n_\xi\rangle>0.
\end{equation}

\subsection{Strong asymptoticity}
A fundamental property at a boundary point $\xi$ of locally finite type of a $\C$-proper convex domain is the strong asymptoticity of complex geodesics endpoint $\xi$, a notion first introduced in \cite{AFGG} (see also \cite{approaching,Arosio-Fiacchi} for further developments).

If $\varphi$ is a complex geodesic with endpoint $\xi$, then by $\tilde \varphi\colon \R_{\geq 0}\to D$ we denote the real geodesic ray $\tilde \varphi=\varphi\circ \gamma_\D$, where
$\gamma_\D\colon \R_{\geq 0}\to \D$ is  the geodesic ray in $\D$ with starting point $0$ and endpoint $1$, that is 
$\gamma_\D(t)={\rm \tanh} (t/2).$
\begin{definition}[Strong asymptoticity]
			Two complex geodesics  $\varphi, \psi\colon \D\to D$ with the same endpoint $\xi$ are {\it strongly asymptotic}
			if there exists
			$T\in \R$  such that
			$$\displaystyle \lim_{t\to +\infty}k_D(\tilde \varphi(t), \tilde\psi(t+T))=0.$$	
\end{definition}

\begin{theorem}[\cite{Arosio-Fiacchi}]\label{strgasy}
Let $D\subset\C^n$ be a $\C$-proper convex domain, and let $\xi\in\partial D$ be a point of locally finite type.   
 Let $\varphi,\psi\colon\D\to D$ be two complex geodesics  with the same endpoint $\xi$.
 Then 
 $$\lim_{t\to +\infty}k_D\left(\tilde \varphi(t), \tilde\psi\left(t+\log\frac{\psi'_N(1)}{\varphi_N'(1)}\right)\right)=0.$$
 In particular any   two complex geodesics  with the same endpoint $\xi$ are strongly asymptotic.
\end{theorem}

\begin{theorem}\cite{Arosio-Fiacchi} \label{strgasysegm}
Let $D\subset\C^n$ be a $\C$-proper convex domain, and let $\xi\in\partial D$ be a point of locally finite type.   
Let $\varphi\colon\D\to D$ be a complex geodesic with $\varphi(1)=\xi$. Then  if $\sigma\colon[t_0,1)\to D$ denotes the inner normal segment 
	$\sigma(t)=\xi+(t-1)\varphi'_N(1)n_\xi$, 
	we have
	$$\lim_{t\to1^-}k_D(\varphi(t),\sigma(t))=0.$$
\end{theorem}
\subsection{Horofunctions and horospheres}
\begin{definition}[Horofunctions and horospheres]
Let $D\subset\C^n$ be a $\C$-proper convex domain, and let $\xi\in \partial D$ be a point of locally finite type. Let $p\in D$. The {\it horofunction 
$h_{\xi,p}\colon D\to \R$ with center $\xi$ and base-point $p$} is defined as $$h_{\xi,p}(z):=\lim_{w\to \xi} [k_D(z,w)-k_D(w,p)].$$
Such a limit exists  as a consequence of Theorem \ref{strgasy}. Since  the function $D\ni z\mapsto k_D(z,w)-k_D(w,p)$ is  $1$-Lipschitz w.r.t. $k_D$, by the Ascoli-Arzel\`a theorem  the convergence is  uniform on compact sets.
Moreover, since the function $D\times D\ni (z,p)\mapsto k_D(z,w)-k_D(w,p)$ is anti-symmetric, it follows that $h_{\xi,\cdot}(\cdot):D\times D\to \R$ is continuous.
Notice that this definition is coherent with the classical one in terms of the horofunction compactification, see e.g. \cite{BH}.
We say that a function $g\colon D\to \R$ is a {\it horofunction with center $\xi$} if there exists $p\in D$ such that $g=h_{\xi,p}.$ Notice that, since
$$h_{\xi,p}(z)=h_{\xi,q}(z)+h_{\xi,p} (q),$$ two horofunctions centered in $\xi$ differ by a constant.

The {\it horosphere} centered in $\xi$ with radius $R>0$ and base-point $p$ is the sublevel set
$$E_p(\xi,R):=\{z\in D\colon h_{\xi,p}(z)<\log R\}.$$
\end{definition}

\begin{remark}\label{Rem:horo-impl-cont}
A  function $g\colon D\to \R$ is a horofunction with center $\xi$ if  and only if
\[
h_{\xi,w}(z)=g(z)-g(w)
\] for all $z,w\in D$.
\end{remark}

\begin{proposition}\cite{Arosio-Fiacchi}\label{horoderivatives}
 Let $\varphi,\psi\colon\D\to D$ be two complex geodesics  with the same endpoint $\xi$.
 Then 
			\begin{equation}\label{formulaT2}
				h_{\xi,\varphi(0)}(\psi(0))=\log\frac{\psi'_N(1)}{\varphi_N'(1)}.
			\end{equation}
			In particular if $\varphi(0)=\psi(0)$ then $\varphi_N'(1)=\psi'_N(1).$
\end{proposition}

\proof
	By  Theorem \ref{strgasy},
\begin{align*}h_{\xi,\varphi(0)}(\psi(0))&=\lim_{t\to+\infty}k_D(\tilde\psi(0),\tilde \varphi(t))-k_D(\tilde\varphi(t),\tilde \varphi (0))\\
&=\lim_{t\to+\infty}k_D\left(\tilde \psi(0),\tilde\psi\left(t+\log\frac{\psi'_N(1)}{\varphi_N'(1)}\right)\right)-k_D(\tilde\varphi(t),\tilde \varphi(0))=\log\frac{\psi'_N(1)}{\varphi_N'(1)}.
\end{align*}
\endproof

A crucial geometric properties of horospheres centered at a point of locally finite type $\xi\in \partial D$ is that they touch the boundary only at $\xi$. 
\begin{proposition}\label{horoonepoint}\cite[Proposition 6.13]{Arosio-Fiacchi}
Let $D\subset\C^n$ be a $\C$-proper convex domain and $\xi\in\partial D$ be a point of locally finite type. Then, for all $R>0$
$$\overline{E_p(\xi,R)}^*\cap\partial^*D=\{\xi\},$$
where $\overline{E_p(\xi,R)}^*$ denotes the closure of $E_p(\xi,R)$ in the one-point compactification of $\C^n$.
\end{proposition}
{Note that  if the domain $D$ is bounded and of (global) finite type, then the previous result follows from the Gromov hyperbolicity of $D$ w.r.t. the Kobayashi distance established by Zimmer \cite{Zim}. }
\subsection{Generalizations of non-tangential limit}
$K$-limits and $K'$-limits are the natural generalizations of non-tangential limits to higher dimensions. We refer to \cite{Abatebook} for an introduction to these concept in several complex variables, and to \cite{AbTau,Arosio-Fiacchi} for the case of  convex  domains of finite type.
In this paper will be mostly interested in $K'$-limits, as $K$-limits will only be relevant in the study of the boundary behaviour of holomorphic maps between $\C$-proper convex domains of finite type (Section \ref{sectiondilation}).

\begin{definition}[$K$-convergence and $K'$-convergence]
Let $D\subset\C^n$ be a $\C$-proper convex domain and let $\xi\in\partial D$ be a point of locally finite type. Let $(z_n)$ be a sequence in $D$ converging to $\xi$. We say that $(z_n)$ \textit{$K$-converges} to $\xi$ if there exists $M>0$ such that $(z_n)$ is contained in the \textit{$K$-region} 
$$K_p(\xi,M):=\{z\in D\colon h_{\xi,p}(z)+k_D(z,p)<2\log M\}.$$
We say that $(z_n)$ \textit{$K'$-converges}  to $\xi$ if it $K$-converges to $\xi$ and if, given a complex geodesic $\varphi\colon\D\to D$ with endpoint $\xi$, there exists a sequence $(\zeta_n)$ in $\D$ converging  to $1\in\partial\D$ such that
$\lim_{n\to+\infty}k_D(z_n,\varphi(\zeta_n))=0.$
\end{definition}
\begin{remark}
Notice that  $(\zeta_n)$  converges  to $1$ non-tangentially.
The definition of $K'$-convergence  does not depend on the choice of the complex geodesic \cite[Proposition~7.8]{Arosio-Fiacchi}. 
Moreover, a sequence $(z_n)$ converging  non-tangentially  to $\xi$ (in the sense of cones) is also $K'$-convergent (see \cite[Corollary 7.15]{Arosio-Fiacchi}).
Finally, the concepts of $K$-convergence and $K'$-convergence can be characterized in a purely extrinsic way through the multitype and the multitype basis (see \cite[Theorem 7.9]{Arosio-Fiacchi}).
\end{remark}
We can now introduce the concepts of $K$-limit and $K'$-limit (sometime also called  {\it restricted $K$-limit}) of a function.
\begin{definition}[$K$-limits and $K'$-limits]
Let $D\subset\C^n$ be a $\C$-proper convex domain and let $\xi\in\partial D$ be a point of locally finite type.  Let $f\colon D\to \C$ be a function and let $a\in \C$. 
We say that  $K\textrm{-}\lim_{z\to\xi}f(z)=a$ if $\lim_{n\to\infty} f(z_n)=a$ for every  sequence $(z_n)$ in $D$ which $K$-converges to $\xi$. 
Similarly, we say that $K'\textrm{-}\lim_{z\to\xi}f(z)=a$ if $\lim_{n\to\infty} f(z_n)=a$  for every sequence $(z_n)$  in $D$ which $K'$-converges to $\xi$.
\end{definition}

\section{Definition of the function $\Omega_\xi$}
In \cite{Arosio-Fiacchi}, as a tool to obtain a Julia--Wolff--Carath\'eodory theorem on convex domains of finite type, the following function was introduced.
		\begin{definition}
			Let $D\subset\C^n$ be a $\C$-proper convex domain and let $\xi\in\partial D$ be a point of locally finite type.
			We define  $\Omega_\xi\colon D\to (-\infty,0)$ as
			$$\Omega_\xi(z)=-\frac{1}{\varphi'_N(1)},$$ where $\varphi\colon \D\to D$ is a complex geodesic such that $\varphi(0)=z$ and with endpoint $\xi$. Note that even though the complex geodesic $\varphi$ might not be unique, $\Omega_\xi(z)$ is well-defined thanks to Proposition \ref{horoderivatives}.
		\end{definition}

		\begin{remark}
		The function $\Omega^\D_1$ is equal to the classical Poisson kernel
		$$\Omega^\D_1(\zeta)=-\frac{1-|\zeta|^2}{|1-\zeta|^2}.$$ 
	It follows from \eqref{ponte}
	 that  in strongly convex domains with $C^\infty$-smooth boundaries the function $\Omega_\xi$ coincides with the pluricomplex Poisson kernel introduced by  Patrizio, Trapani and the
	 second author \cite{BP,BPT}, and that  in strongly  convex domains with $C^3$-smooth boundary  $\Omega_\xi$ coincides with the pluricomplex Poisson kernel  introduced by Huang--Wang \cite{HuWa}.
	 \end{remark}
	
The next result  follows immediately from Proposition \ref{horoderivatives} (see also Remark~\ref{Rem:horo-impl-cont}).
		\begin{proposition}[Poisson-horofunction formula]\label{omegahorofunction}
		Let $D\subset\C^n$ be a $\C$-proper convex domain and let $\xi\in\partial D$ be a point of locally finite type.
			Then  $-\log|\Omega_\xi(\cdot)|$ is a horofunction centered at $\xi$, that is 
			for all $z,w\in D$,
			\begin{equation}\label{orosferepoisson}
				h_{\xi,z}(w)=\log |\Omega_\xi(z)|-\log  |\Omega_\xi(w)|.
			\end{equation}
			As a consequence $\Omega_\xi$ is continuous and its level sets are exactly  horospheres centered in $\xi$.
		\end{proposition}
		As a corollary one immediately obtains the behaviour of $\Omega_\xi$ on complex geodesics.
		
		\begin{corollary}\label{forPoigeo}
			Let $D\subset\C^n$ be a $\C$-proper convex domain and let $\xi\in\partial D$ be a point of locally finite type. Let $\varphi\colon \D\to D$ be a complex geodesic with endpoint $\xi$.
			Then, for all $\zeta\in \D$,
			$$\Omega_\xi^D(\varphi(\zeta))= \frac{\Omega_1^\D(\zeta)}{\varphi'_N(1)}.$$	
		\end{corollary}
		\begin{proof}
			By \eqref{orosferepoisson}, for all $\zeta\in \D$,
			$$-\log|\Omega^\D_1(\zeta)|= h^\D_{1,0}(\zeta)=h^D_{\xi,\varphi(0)}(\varphi(\zeta))=\log|\Omega^D_\xi(\varphi(0))|-\log|\Omega^D_\xi(\varphi(\zeta))|,$$
			and we are done.
		\end{proof}

From the previous corollary and the explicit description of the complex geodesics in \cite{JPZ}, we can compute $\Omega_\xi$ in the case of smooth complex ellipsoids.

\begin{example}[Complex ellipsoid]\label{ellipsoid}
	Let $m:=(m_1,\cdots,m_{n-1})$ an array of even integers, and let
	$$\mathbb{E}_m:=\left\{z\in\C^n: |z_0|^2+\sum_{j=1}^{n-1}|z_j|^{m_j}<1 \right\}.$$
	If $\xi=(1,0)\in\C\times\C^{n-1}$, then
	$$\Omega^{\mathbb{E}_m}_\xi(z)=-\frac{1-|z_0|^2-\sum_{j=1}^{n-1}|z_j|^{m_j}}{|1-z_0|^2}.$$
\end{example}

\section{$-\Omega_\xi$ is the normal derivative of the pluricomplex Green function} In this section we prove that the  normal derivative at $\xi$ of the pluricomplex Green function $G_z$ exists and equals $-\Omega_\xi(z)$.
We first need to introduce the pluricomplex Green function for (non necessarily bounded) $\C$-proper convex domains.
\begin{proposition}
Let $D\subset \C^n$ be a $\C$-proper convex domain and let $w\in D$.
 The Monge--Amp\`ere equation
\begin{equation}\begin{cases}
u\in {\rm Psh}(D)\cap L^\infty_{\rm{loc}}( D\setminus \{w\}),\\
(dd^cu)^n=0 \quad\textrm{on} \quad D\setminus \{w\},\\
u(z)\stackrel{z\to\eta}\longrightarrow 0 \ \mbox{ if }\eta\in\partial^* D,\\
u - \log|z-w| =O(1)\quad\textrm{as} \quad z\to w,
\end{cases}
\end{equation}
has a unique continuous solution $G_w\colon  \overline D^*\to [-\infty,0]$ which is defined as
\begin{equation}\label{GinCconvex}
G_w(z):=\log\tanh (k_D(z,w)/2), \quad z\in D.
\end{equation}
\end{proposition}
\begin{proof} Recall that  $c_D=k_D$  by (2) in Remark \ref{remarkgeodesic}.
The function $G_w$ is clearly continuous and harmonic on every complex geodesic, and is  plurisubharmonic since it is equal to $z\mapsto \log\tanh (c_D(z,w)/2)$.
Thanks to  Proposition \ref{maximaldisk},  the function $G_w\colon D\to [-\infty,0)$ is  maximal plurisubharmonic.
The boundary behaviour is easily verified. Uniqueness can be shown as in \cite[Theorem 7.4]{BracciTrapani}.
\end{proof}

\begin{theorem}\label{derivataGreen}
Let $D\subset\C^n$ be a $\C$-proper convex domain and let  $\xi\in\partial D$ be a point of locally finite type. Then for all $z\in D$ the limit
$$\lim_{t\to0^-}\frac{1}{t}G_z(\xi+tn_\xi)=:\frac{\partial G_{z}}{\partial n_\xi}(\xi)$$ exists and equals
$-\Omega_\xi(z).$ Moreover the convergence is
 uniform on compact sets.
\end{theorem}
\proof
Let $\sigma(t)=\xi-(1-t)n_\xi$, we need to show that 
$$\lim_{t\to1^-}\frac{G_z(\sigma(t))}{1-t}=\Omega_\xi(z)$$
uniformly on compact sets. Let $\varphi:\D\to D$ be a complex geodesic with endpoint $\xi$ and  $\varphi'_N(1)=1$. 

We have 
\begin{equation}\label{prodottone}
\frac{G_z(\sigma(t))}{1-t}=\frac{G_z(\sigma(t))}{-2e^{-k_D(z,\sigma(t))}}\cdot\frac{-2e^{-k_D(z,\sigma(t))}}{-2e^{-k_D(z,\varphi(t))}}\cdot\frac{-2e^{-k_D(z,\varphi(t))}}{1-t}.
\end{equation}
We now study the limit of the three factors on the right hand side of \eqref{prodottone}.

By (\refeq{GinCconvex}) we have
$$G_z(\sigma(t))=\log\left(\frac{1-e^{-k_D(z,\sigma(t))}}{1+e^{-k_D(z,\sigma(t))}}\right)=\log\left(1-\frac{2e^{-k_D(z,\sigma(t))}}{1+e^{-k_D(z,\sigma(t))}}\right),$$ hence 
$$\lim_{t\to1^-}\frac{G_z(\sigma(t))}{-2e^{-k_D(z,\sigma(t))}}=1$$ uniformly on compact sets. 
Now by
  Theorem \ref{strgasysegm}
we have 
$\lim_{t\to1^-}k_D(\varphi(t),\sigma(t))=0$, and thus
$$|k_D(z,\varphi(t))-k_D(z,\sigma(t))|\leq k_D(\varphi(t),\sigma(t))\to0.$$
Hence  $$\lim_{t\to1^-}\frac{-2e^{-k_D(z,\sigma(t))}}{-2e^{-k_D(z,\varphi(t))}}=1,$$ uniformly.
Notice that 
$$\lim_{t\to1^-}k_D(z,\varphi(t))-k_\D(0,t)=\lim_{t\to1^-}k_D(z,\varphi(t))-k_D(\varphi(0),\varphi(t))=h_{\xi,\varphi(0)}(z),$$
uniformly on compact subsets. Hence, by the Poisson-horofunction formula (Proposition \ref{omegahorofunction}),
$$\frac{-2e^{-k_D(z,\varphi(t))}}{1-t}=e^{-k_D(z,\varphi(t))+k_\D(0,t)}
\cdot \frac{-2e^{-k_\D(0,t)}}{1-t}
\stackrel{t\to 1^-}\longrightarrow -e^{-h_{\xi,\varphi(0)}(z)}=\Omega_\xi(z),$$
uniformly on compact subsets.
\endproof

As a corollary we obtain a generalization of Abate--Venturini's formula to this setting.
\begin{corollary}\label{newvesentini}
Let $D\subset\C^n$ be a $\C$-proper convex domain and let  $\xi\in\partial D$ be a point of locally finite type. 
	Then for all $z,w\in D$ we have
	$$h_{\xi,w}(z)=\log\left(\frac{\partial G_{w}}{\partial n_\xi}(\xi)\right)-\log\left(\frac{\partial G_{z}}{\partial n_\xi}(\xi)\right).$$
\end{corollary}

\begin{remark}\label{kobestnormal}
The proof of Theorem \ref{derivataGreen} yields the following estimate of $k_D$ in the normal direction: for all $p\in D$,
$$\lim_{t\to0^+}[k_D(\xi-tn_\xi,p)+\log t]=-\log\frac{|\Omega_\xi(p)|}{2}.$$
\end{remark}

\section{$\Omega_\xi$ solves the Monge-Amp\`ere equation}
Let $D\subset\C^n$ be a $\C$-proper convex domain and let  $\xi\in\partial D$ be a point of locally finite type.  		
 In this section we will prove that the function $\Omega_\xi$ is a continuous solution to the following homogeneous Monge-Amp\`ere equation: 
 \begin{equation}\label{poissonnonintro}\begin{cases*}
u\in {\rm Psh}(D)\cap L^\infty_{\rm{loc}}( D), \\
 u<0  \mbox{ on }D,\\
(dd^cu)^n=0 \quad\textrm{on} \quad D,\\
				u(z)\stackrel{z\to\eta}\longrightarrow 0 \ \mbox{ if }\eta\in\partial^* D\backslash\{\xi\}\\
				u(z)\approx -\|z-\xi\|^{-1} \mbox{ as }z\to\xi\mbox{ non-tangentially},
			\end{cases*}
		\end{equation}
and thus  deserves  the name {\it pluricomplex Poisson kernel} of $D$ with singularity at $\xi$.

First of all, notice that $\Omega_\xi$ takes only strictly negative values by construction, and that it is continuous thanks to Proposition \ref{omegahorofunction}.

\begin{proposition}\label{maximal}
	Let $D\subset\C^n$ be a $\C$-proper convex domain and let  $\xi\in\partial D$ be a point of locally finite type. Then the function $\Omega_\xi$ is a maximal plurisubharmonic function.
\end{proposition}
\proof
	First of all we show that $\Omega_\xi$ is plurisubharmonic. By Theorem \ref{derivataGreen} and the symmetry of $G$, 
	$$\Omega_\xi(z)=\lim_{t\to0^+}\frac{G_z(\xi-tn_\xi)}{t}=\lim_{t\to0^+}\frac{G_{\xi-tn_\xi}(z)}{t},$$ where the convergence is uniform on compact subsets. Since the functions  $\frac{1}{t}G_{\xi-tn_\xi}$ are plurisubharmonic, it follows that  $\Omega_\xi$ is also plurisubharmonic.
	
	As for the maximality,
by Corollary \ref{forPoigeo}, for every complex geodesic $\varphi\colon \D\to D$ with $\varphi(1)=\xi$, we have
	$\Omega_\xi\circ\varphi=\Omega_1^\D/\varphi'_N(1),$
	which is a harmonic function on $\D$. By Proposition~\ref{maximaldisk}, $\Omega_\xi$ is maximal.
\endproof

Next, we study the behavior of $\Omega_\xi$ at the boundary.

\begin{proposition}
	Let $D\subset\C^n$ be a $\C$-proper convex domain and let  $\xi\in\partial D$ be a point of locally finite type.  If $\eta\in\partial^* D\backslash\{\xi\}$, then
	$\lim_{z\to\eta}	\Omega_\xi(z)=0.$
\end{proposition}
\proof	
By the Poisson-horofunction formula (Proposition \ref{omegahorofunction}), it is enough to show that, given a base-point $p\in D$ and a sequence $(z_n)$ converging to $\eta$, we have that $h_{\xi,p}(z_n)$ converges to $+\infty$.
This follows at once from Proposition \ref{horoonepoint}.
\endproof

Now, we analyze the behavior at the pole $\xi$. 
We  show   that $\Omega_\xi$ blows up on  sequences $K'$-converging to $\xi$ like the Poisson kernel of the half-plane  $\H:=\{\zeta\in\C:\Re\zeta<0\}$. 
Let $P\colon D\to \H$ be the orthogonal projection $P(z)=\langle z-\xi,n_\xi\rangle$, and   recall that the Poisson kernel of the left half-plane with pole at $0$ is given by
\begin{equation}\label{poissonsemipiano}\Omega_0^\H(\zeta)=2\frac{\Re\zeta}{|\zeta|^2}=2\Re\left(\frac{1}{\zeta}\right),\quad\zeta\in\H.
\end{equation}

\begin{proposition}\label{OmegaDvsH}
	Let $D\subset\C^n$ be a $\C$-proper convex domain and let  $\xi\in\partial D$ be a point of locally finite type. Then
	$$ K'\textrm{-}\lim_{z\to\xi}\frac{\Omega_\xi(z)}{\Omega_0^\H(\langle z-\xi,n_\xi\rangle)}=1.$$
\end{proposition}
Before proving Proposition \ref{OmegaDvsH} we need to recall some basic facts about angular derivatives of holomorphic functions.
\begin{definition}[Angular derivative]
Let $f\colon\D\to \C$ be a  holomorphic function. Let $\xi\in\partial\D$. We say that $f$ has finite \textit{angular derivative} $f'(\xi)$ at $\xi\in\partial\D$ if $f(\xi):=\angle\lim_{\zeta\to\xi}f(\zeta)$ exists finitely and 
$$f'(\xi):=\angle\lim_{\zeta\to\xi}\frac{f(\xi)-f(\zeta)}{\xi-\zeta}$$
exists finitely. Where, as customary, $\angle\lim$ denotes the  non-tangential limit. 
\end{definition}
By a version of the classical Julia-Wolff-Carath\'eodory's theorem (see, {\sl e.g.}, \cite[Theorem 1.7.2]{booksemigroup}), a holomorphic function $f\colon\D\to\C$ 
has finite angular derivative at $\xi\in\partial\D$ if and only if $\angle \lim_{\zeta\to\xi}f'(\zeta)$ exists finitely. If this is the case, 
$$f'(\xi)=\angle \lim_{\zeta\to\xi}f'(\zeta).$$

\begin{remark} 
Let $D\subset\C^n$ is a $\C$-proper convex domain and $\xi\in\partial D$ be a point of locally finite type. Let $\varphi\colon\D\to D$ be a complex geodesic with endpoint $\xi$. Then $\varphi_N'(1)$ is the angular derivative at $1\in\partial\D$ of the holomorphic function $\varphi_N\colon \D\to\H$ given by $\varphi_N:=P \circ \varphi=\langle \varphi-\xi,n_\xi\rangle$.
\end{remark}

\begin{remark}\label{ciserve}
Again by the Julia-Wolff-Carath\'eodory's theorem, 
a holomorphic function $f\colon\D\to\D$ 
has finite angular derivative at $\xi\in\partial\D$ if and only if $\alpha:=\liminf_{\zeta\to\xi}\frac{1-|f(\zeta)|}{1-|\zeta|}<+\infty$, and in this case
\begin{equation}\label{ntdildisk}
\alpha=\angle\lim_{\zeta\to\xi}\frac{1-|f(\zeta)|}{1-|\zeta|}=|f'(\xi)|.
\end{equation}
\end{remark}

\begin{lemma}\label{omegaratiodisk}
Let $f\colon\D\to\D$ be a holomorphic function and  let $\xi\in\partial\D$. Assume that $f$ has finite angular derivative at $\xi\in\partial\D$ and $f(\xi)\in\partial\D$, then
\begin{equation}\label{ntpoidisk}
\angle\lim_{\zeta\to\xi}\frac{\Omega^\D_\xi(\zeta)}{\Omega^\D_{f(\xi)}(f(\zeta))}=|f'(\xi)|.
\end{equation}
\end{lemma}
\proof
We have
$$\frac{\Omega^\D_\xi(\zeta)}{\Omega^\D_{f(\xi)}(f(\zeta))}=\frac{1-|\zeta|^2}{|\xi-\zeta|^2}\cdot\frac{|f(\xi)-f(\zeta)|^2}{1-|f(\zeta)|^2}=\frac{1+|\zeta|}{1+|f(\zeta)|}\cdot\frac{1-|\zeta|}{1-|f(\zeta)|}\cdot\left|\frac{f(\xi)-f(\zeta)}{\xi-\zeta}\right|^2.$$
As $\zeta$ tends to $1$ non-tangentially, the  first term tends to 1, the second term tends to $\frac{1}{|f'(\xi)|}$ by  (\refeq{ntdildisk}), and the third term tends to $|f'(\xi)|^2$.
\endproof 

\begin{remark}\label{omegaratiodiskH}
If $f\colon\D\to\H$ is a holomorphic function with finite angular derivative at $\xi\in\partial\D$ and $f(\xi)\in\partial\H$, then it immediately follows from (\refeq{ntpoidisk}) that
$$\angle\lim_{\zeta\to\xi}\frac{\Omega^\D_\xi(\zeta)}{\Omega^\H_{f(\xi)}(f(\zeta))}=|f'(\xi)|.$$
\end{remark}

\proof[Proof of Proposition \ref{OmegaDvsH}]
Let $\varphi\colon\D\to D$ be a complex geodesic with endpoint $\xi$.
Let $(z_n)$ be a sequence  $K'$-converging to $\xi$.  By definition there exists a sequence $(\zeta_n)$ in $\D$ such that $\zeta_n\to1$ non-tangentially and
$k_D(z_n,\varphi(\zeta_n))\to0.$
Notice that by the Poisson-horofunction formula (Proposition \ref{omegahorofunction}) and since $|h_{\xi,p}(z)|\leq k_D(z,p)$ for all $z,p\in D$, we have
$$\left|\log\frac{\Omega_\xi^D(z_n)}{\Omega_\xi^D(\varphi(\zeta_n))}\right|=\left|h^D_{\xi,z_n}(\varphi(\zeta_n))\right|\leq k_D(z_n,\varphi(\zeta_n))\to0,$$
thus
$$\frac{\Omega_\xi^D(z_n)}{\Omega_\xi^D(\varphi(\zeta_n))}\to1.$$
Define $\varphi_N\colon\D\to\H$  by $\varphi_N:=P \circ \varphi=\langle \varphi-\xi,n_\xi\rangle$.
 Similarly as before, using the non-expansivity of the Kobayashi distance,
\begin{equation*}
\begin{split}\left|\log\frac{\Omega_0^\H(\langle z_n-\xi,n_\xi\rangle)}{\Omega_0^\H( \varphi_N(\zeta_n))}\right|&=|h^\H_{0,{\langle z_n-\xi,n_\xi\rangle}}(\varphi_N(\zeta_n))|\\&\leq k_\H(\langle z_n-\xi,n_\xi\rangle,\varphi_N(\zeta_n))\leq k_D(z_n,\varphi(\zeta_n))\to0,
\end{split}
\end{equation*}
so that
$$\frac{\Omega_0^\H(\langle z_n-\xi,n_\xi\rangle)}{\Omega_0^\H( \varphi_N(\zeta_n))}\to1.$$
Hence, it is enough  to show
$$\lim_{n\to+\infty}\frac{\Omega^D_\xi(\varphi(\zeta_n))}{\Omega_0^\H(\varphi_N(\zeta_n))}=1.$$
By Corollary \ref{forPoigeo}, $$\Omega^D_\xi(\varphi(\zeta_n))=\frac{\Omega^\D_1(\zeta_n)}{\varphi'_N(1)},$$
and, by Remark \ref{omegaratiodiskH}, we have
$$\frac{\Omega_1^\D(\zeta_n)}{\Omega_0^\H(\varphi_N(\zeta_n))}\to\varphi_N'(1),$$
hence
$$\lim_{n\to+\infty}\frac{\Omega^D_\xi(\varphi(\zeta_n))}{\Omega_0^\H( \varphi_N(\zeta_n))}=\lim_{n\to+\infty}\frac{\Omega^D_\xi(\varphi(\zeta_n))}{\Omega^\D_1(\zeta_n)}\frac{\Omega^\D_1(\zeta_n)}{\Omega_0^\H(\varphi_N(\zeta_n))}=1.$$
\endproof
We are now able to describe the asymptotic behaviour of $\Omega_\xi$ on a family of curves with endpoint $\xi$.
\begin{corollary}\label{lemmaomegacurv}
Let $D\subset\C^n$ be a $\C$-proper convex domain and let $\xi\in\partial D$ be a point of locally finite  type. 
Let  $\gamma\colon[0,1)\to D$ be a curve  which $K'$-converges to $\xi$,   such that 
$$\gamma'_N(1):=\lim_{t\to1^-}\frac{\langle \xi-\gamma(t),n_\xi\rangle}{1-t}$$ exists finitely and is nonzero.  
Then
$$\lim_{t\to1^-}\Omega_{\xi}(\gamma(t))(1-t)=-\Re\frac{2}{\gamma_N'(1)}.$$
\end{corollary}

\proof
By \eqref{poissonsemipiano} we have
$$\Omega_0^\H(\langle\gamma(t)-\xi,n_\xi\rangle)(1-t)=2\Re\left(\frac{1-t}{\langle\gamma(t)-\xi,n_\xi\rangle}\right)\stackrel{t\to 1-}\longrightarrow -\Re\frac{2}{\gamma_N'(1)},$$
hence the result follows by  Proposition \ref{OmegaDvsH}.
\endproof

We have thus shown that $\Omega_\xi$ solves the Monge–Ampère equation~(\refeq{poissonnonintro}).

\begin{remark}
Notice that by (\refeq{varphi'_N(1)}) any complex geodesic with endpoint $\xi$ is  a curve satisfying the assumptions of the previous corollary when restricted to $[0,1).$
\end{remark}

\begin{remark}
The previous corollary does not hold if we only assume $K$-convergence of the curve $\gamma$ instead of $K'$-convergence.
	Take for example $D=\B^2$, $\lambda\in\D$ and $\gamma_\lambda:[0,1)\to\B^2$ given by $\gamma_\lambda(t)=(t,\lambda\sqrt{1-t^2})$. The curve $\gamma_\lambda$ has endpoint $e_1$ and $(\gamma'_\lambda)_N(1)=1$, but	$$\lim_{t\to1^-}\Omega^{\B^2}_{e_1}(\gamma(t))(1-t)=-2(1-|\lambda|^2).$$
\end{remark}

Arguing similarly as in \cite[Theorem 7.1 and Proposition 7.4]{BPT} one can prove the following

\begin{proposition}
	Let $D\subset\C^n$ be a $\C$-proper convex domain and let  $\xi\in\partial D$ point of locally finite  type. Let $u\colon D\to\R$ be a maximal plurisubharmonic function such that $\lim_{z\to\eta}u(z)=0$ for all $\eta\in\partial^* D\backslash\{\xi\}$.
\begin{itemize}
\item If  $\lim_{z\to\xi}\frac{u(z)}{\Omega_\xi(z)}=1,$ then $u\equiv\Omega_\xi$.
\item If  $u$ has the same level sets of $\Omega_\xi$ then there exists $c>0$ such that $u=c\Omega_\xi$.
\end{itemize}
\end{proposition}

We end this section proving a generalized Phragmen--Lindel\"of theorem.
Let $\Gamma_\xi$ be the family of $C^1$-smooth curves $\gamma\colon [0,1)\to D$ 
which $K'$-converge to $\xi$,   such that $\gamma'_N(1):=\lim_{t\to1^-}\langle \gamma'(t),n_\xi\rangle$ exists, and is a strictly positive real number.
Notice that by (\refeq{varphi'_N(1)}) any complex geodesic with endpoint $\xi$ belongs to $\Gamma_\xi$ when restricted to $[0,1)$. 
Let $\mathscr{F}_\xi$ be the following family of functions 
\begin{equation}\label{MAmax}\begin{cases*}
		u\mbox{ is plurisubharmonic in }D,\\
		u<0 \ \mbox{ in }D,\\
		\limsup_{t\to1^-}u(\gamma(t))(1-t)\leq-\frac{2}{\gamma'_N(1)}, \ \ \forall \gamma\in\Gamma_\xi.
	\end{cases*}
\end{equation}

\begin{proposition}\label{maximalelement}
Let $D\subset\C^n$ be a $\C$-proper convex domain and let $\xi\in\partial D$ point of locally finite  type. 
Then $\Omega_\xi$ is the greatest element of the family $\mathscr{F}_\xi$.
	
\end{proposition}
\proof
By Corollary \ref{lemmaomegacurv} $\Omega_\xi$  belongs to the family $\mathscr{F}_\xi$.
Let $u\in \mathscr{F}_\xi$, and let  $z\in D$.  Let $\varphi\colon\D\to D$ be a complex geodesic with $\varphi(0)=z$ and $\varphi(1)=\xi$.
Notice that the function $u\circ\varphi$ is  subharmonic on $\D$ and strictly negative. Since  the curve $t\mapsto\varphi(t)$ belongs to  the family $\Gamma_\xi$, we have
$$\limsup_{t\to1^-}u(\varphi(t))(1-t)\leq-\frac{2}{\varphi'_N(1)}=2\Omega_\xi(z).$$
By the classical Phragmen-Lindel\"of theorem (see e.g. \cite[Lemma 5.2]{BPT}),  for all $\zeta\in\D$
$$u(\varphi(\zeta))\leq|\Omega_\xi(z)|\Omega^\D_1(\zeta).$$
In particular if $\zeta=0$ we have
$u(z)\leq \Omega_\xi(z).$
\endproof

\section{A new estimate for the Kobayashi distance}

 In this section we  prove the following asymptotic estimate of the Kobayashi distance in terms of the pluricomplex Poisson kernel, which appears to be new also in the case of strongly convex domains.
Beside being interesting  in its own right, this estimate is the crucial ingredient (together with Demailly's theory \cite{Demailly}) in the proof of the  reproducing formula in Section \ref{sectDem}.

\begin{theorem}\label{Kobestimate}Let $D\subset\C^n$ be a $\C$-proper convex domain and let  $\xi\in\partial D$ be a point of locally finite type. Let $p\in D$, then
	
	$$\lim_{z\to\xi}[k_D(z,p)+\log\delta_D(z)]=-\log\frac{|\Omega_\xi(p)|}{2}.$$
\end{theorem}
 \begin{remark}
Notice that, if one knows that for all $z\in D$ the limit $$h(z):=\lim_{x\to\xi}[k_D(x,z)+\log\delta_D(x)]$$ exists, then it is immediate to see that the function $h\colon D\to \R$ is a horofunction centered at $\xi$. Indeed $$h(z)-h(w)=\lim_{x\to\xi}[k_D(z,x)+\log\delta_D(x)-k_D(w,x)-\log\delta_D(x)]=h_{\xi,w}(z).$$
\end{remark}
We start by showing that  $\Omega_\xi$ is continuous  with respect to the pole $\xi$.  	In the bounded $C^\infty$-smooth strongly convex case, the function $(z,\xi)\mapsto \Omega_\xi(z)$ is actually 	$C^\infty$-smooth on $ D\times \partial D$
thanks to regularity of the boundary spherical representation \cite{CHL,BPT}.  
\begin{proposition}\label{contpole}
	Let $D\subset\C^n$ be a $\C$-proper convex domain and let  $\xi_0\in\partial D$ be a point of locally finite type. Then there exists a neighborhood $U$ of $\xi_0$  such that
	$(z,\xi)\mapsto \Omega_\xi(z)$ is a continuous function on $ D\times  (\partial D\cap U)$.

\end{proposition}
We first need a lemma.

\begin{lemma}\label{liminfdila}
Let $(f_n\colon\D\to\D)$ be a sequence of holomorphic functions converging uniformly on compact subsets to $f\colon \D\to\D$. Assume that there exist $\xi,\eta\in\partial\D$ such that for all $n\geq 0$, $\angle \lim_{\zeta\to\xi }f_n(\zeta)=\eta$ with finite angular derivative $f_n'(\xi)$. Assume that $$\liminf_{n\to+\infty}|f_n'(\xi)|<+\infty.$$ Then $\angle \lim_{\zeta\to\xi}f(\zeta)=\eta$, $f$ has finite angular derivative at $\xi$, and
$$|f'(\xi)|\leq\liminf_{n\to+\infty}|f_n'(\xi)|.$$
\end{lemma}
\proof
Up to conjugation, we may assume that $\xi,\eta=1$, then by the classical Julia-Wolff-Carath\'eodory Theorem  we have that $f_n'(1)$ are real positive numbers.  By  Julia's  Lemma (see, {\sl e.g.}, \cite[Theorem~1.4.7]{BCD}) we have, for all $n\geq 0$ and $\zeta\in\D$,
$$\frac{|1-f_n(\zeta)|^2}{1-|f_n(\zeta)|^2}\leq f'_n(1)\frac{|1-\zeta|^2}{1-|\zeta|^2}.$$
Let $(n_k)$ be a sequence converging to $\infty$ such that  $f_{n_k}'(1)\to a:=\liminf_{n\to+\infty}f'_n(1)<+\infty$. Then passing to the limit we have for all $\zeta\in\D$
$$\frac{|1-f(\zeta)|^2}{1-|f(\zeta)|^2}\leq a\frac{|1-\zeta|^2}{1-|\zeta|^2},$$
which implies that $\angle \lim_{\zeta\to1}f(\zeta)=1$.
Notice that $\frac{|1-f(\zeta)|^2}{1-|f(\zeta)|^2}\geq\frac{1-|f(\zeta)|}{1+|f(\zeta)|}$ for all $\zeta\in \D$. Thus for  $t\in(0,1)$
$$\frac{1-|f(t)|}{1-t}\leq a\frac{1+|f(t)|}{1+t}.$$
Hence $\liminf_{\zeta\to1^-}\frac{1-|f(\zeta)|}{1-|\zeta|}\leq a$, so by Remark \ref{ciserve} the function $f$ has finite angular derivative at $1$ and $f'(1)\leq a$.
\endproof

\begin{proof}[Proof of Proposition \ref{contpole}] 
By definition there exist $L\geq 2$ and  a neighborhood $U$ of $\xi$ such that $\partial D\cap U$ is of class $C^L$ and every point in $\partial D\cap U$ has line type at most $L$. We need to show that if $z_n\to z$ and $\xi_n\to\xi\in  \partial D\cap U$, then
$\Omega_{\xi_n}(z_n)\to\Omega_{\xi}(z).$
For all $n\geq 0$, let $\varphi_n\colon\D\to D$ be a complex geodesic such that $\varphi_n(0)=z_n$ and $\varphi_n(1)=\xi_n$, then by definition 
$$\Omega_{\xi_n}(z_n)=-\frac{1}{(\varphi_n)_N'(1)}.$$
By tautness we can find a subsequence $(n_k)$ such that $(\varphi_{n_k})$ converges uniformly on compact sets to a complex geodesic $\varphi\colon\D\to D$ such that $\varphi(0)=z$. Moreover $\varphi(1)=\xi$ (see for example the proof of \cite[Proposition 4.3]{Arosio-Fiacchi}).

Recall that $(\varphi_{n_k})_N'(1)$ is the angular derivative at $1$ of $\D\ni\zeta\mapsto \langle \varphi_{n_k}(\zeta)-\xi_{n_k}, n_{\xi_{n_k}}\rangle\in\H$, then post-composing with the inverse of the Cayley transform
\begin{equation}\label{cayley}\mathscr{C}\colon \D\to\H,\quad \mathscr{C}(\zeta)=\frac{\zeta-1}{\zeta+1},
\end{equation}
and using Lemma \ref{liminfdila} we obtain  $\liminf_{k\to+\infty}(\varphi_{n_k})_N'(1)\geq \varphi_N'(1)$, which implies  $$\liminf_{n\to+\infty}\Omega_{\xi_n}(z_n)\geq \Omega_{\xi}(z).$$

 For all $n\geq 0$, let $\tilde\rho_n:D\to \D$ be  a holomorphic left inverse of $\varphi_n$. Since $\partial D\cap U$ is $C^2$-smooth, we can find $N\geq 0$ and $r>0$ such that 
$B(\xi_n-rn_{\xi_n},r)\subseteq D$ for all $n\geq N$.
For all $n\geq N$ consider the holomorphic function $f_n\colon\D\to\D$ defined by 
$$f_n(\zeta):=\tilde\rho_n(\xi_n+rn_{\xi_n}(\zeta-1)).$$
Notice that by tautness, up to extracting subsequences, $(f_n)$ converges uniformly on compact sets to the function $f\colon \D\to \D$ given by $f(\zeta)=\tilde\rho(\xi+rn_{\xi}(\zeta-1))$, where $\tilde\rho:D\to \D$ is a left inverse of $\varphi$.
By \cite[Theorem~1.3 and Lemma~11.12]{Arosio-Fiacchi}, we have for all $n\geq 0$,
$$K'\textrm{-}\lim_{z\to \xi_n}(d\tilde \rho_{n})_z(n_{\xi_n})=|\Omega_{\xi_n}(z_n)|.$$
Thus, the angular derivative of $f_n$ at $1$ is $r|\Omega_{\xi_n}(z_n)|$. Similarly we obtain that  the angular derivative of $f$ at $1$ is $r|\Omega_{\xi}(z)|$.
By Lemma \ref{liminfdila},
$$\limsup_{n\to+\infty}\Omega_{\xi_n}(z_n)\leq \Omega_{\xi}(z),$$
and we are done.
\end{proof}

Let $D\subset\C^n$ be a domain, and let $\xi_0\in\partial D$ be a $C^2$-smooth boundary point. 
Choose a   neighborhood  $V$  of $\xi$ and   $r>0$  such that 
 $B(\xi-rn_\xi,r)\subseteq D$ for all $\xi\in\partial D\cap V$. 
 If a neighborhood $U$  of $\xi_0$ is small enough, then    for all $z\in D\cap U$ there exists a unique  point $\pi(z)\in \partial  D\cap V$ such that $\|z-\pi(z)\|=\delta_D(z)$ (see, {\sl e.g.}, \cite[Lemma 2.1]{BaBo}).

\begin{lemma}\label{lemmaestimate}
Let $D\subset\C^n$ be a domain, and let $\xi_0\in\partial D$ be a $C^2$-smooth boundary point.
Let $U, r>0$ be as above. If $z,w\in D\cap U$ satisfy $\pi(z)=\pi(w)$  and $\max\{\delta_D(z),\delta_D(w)\}<r$, then
$$k_D(z,w)\leq\left|\log\frac{\delta_D(z)}{\delta_D(w)}\right|-\log\left(1-\frac{\max\{\delta_D(z),\delta_D(w)\}}{2r}\right).$$
\end{lemma}
\proof

Denote $\pi(z)=\pi(w)=:\hat\xi\in\partial D\cap V$.  We may assume $\delta_D(z)\geq\delta_D(w)$. Denote with $B:=B(\hat\xi-rn_{\hat\xi},r)$, clearly we have
\begin{align*}k_D(z,w)\leq k_B(z,w)&=\log\left(\frac{\delta_D(z)}{\delta_D(w)}\cdot\frac{2r-\delta_D(w)}{2r-\delta_D(z)}\right)\\&\leq \log\frac{\delta_D(z)}{\delta_D(w)}+\log\left(\frac{2r}{2r-\delta_D(z)}\right)\\&=\log\frac{\delta_D(z)}{\delta_D(w)}-\log\left(1-\frac{\delta_D(z)}{2r}\right).\end{align*}

\endproof

\begin{lemma}\label{lemmaexclaim}
Let $D\subset\C^n$ be a $\C$-proper convex domain and let  $\xi\in\partial D$ be a point of locally finite type.
Let $\varphi\colon\D\to D$ be a complex geodesic with endpoint $\xi$. If $(\zeta_n)$ is a sequence in $\D$ converging non-tangentially to $1$, then
$$\limsup_{n\to+\infty}\frac{\delta_D(\varphi(\zeta_n))}{1-|\zeta_n|}\leq\varphi'_N(1). $$
\end{lemma}
\proof
Let $P:\C^n\to\C$ be the projection $P(z)=\langle z-\xi,n_\xi\rangle$. By convexity $P(D)\subset\H$.
Set $f\colon\D\to\D$ given by $f:=\mathscr{C}^{-1}\circ P\circ\varphi$, where $\mathscr{C}$ is the Cayley transform \eqref{cayley}. 
Notice that $\angle \lim_{\zeta\to1}f(\zeta)=1$ and $f'(1)=2\varphi'_N(1)$ (since $(\mathscr{C}^{-1})'(0)=2$).
By (\refeq{ntdildisk}) we have
$$\angle\lim_{\zeta\to1}\frac{1-|f(\zeta)|}{1-|\zeta|}=2\varphi'_N(1).$$
Finally, noticing that for all $\zeta\in\D$
$$\delta_\H(\mathscr{C}(\zeta))=-\Re\mathscr{C}(\zeta)=\frac{1-|\zeta|^2}{|1+\zeta|^2},$$
we have
$$\frac{\delta_D(\varphi(\zeta_n))}{1-|\zeta_n|}\leq\frac{\delta_\H( P(\varphi(\zeta_n)))}{1-|\zeta_n|}=\frac{1-|f(\zeta_n)|}{1-|\zeta_n|}\cdot\frac{1+|f(\zeta_n)|}{|1+f(\zeta_n)|^2}\stackrel{n\to\infty}\longrightarrow\varphi'_N(1).$$
\endproof

\begin{remark}
We will prove later on that actually 
 $\angle\lim_{\zeta\to1}\frac{\delta_D(\varphi(\zeta))}{1-|\zeta|}=\varphi'_N(1),$ see Theorem~\ref{deltaratio}. Notice however that the proof of Theorem \ref{deltaratio} is based on Theorem \ref{Kobestimate}.
\end{remark}

\proof[Proof of Theorem \ref{Kobestimate}]

\textbf{Lower bound:}
Let $(z_n)$ be a sequence in $D$ converging to $\xi$. Let $\xi_n:=\pi(z_n)\in\partial D$ be the closest boundary point.
Fix $\varepsilon>0$.
By Remark \ref{kobestnormal}, for all $n$ large enough  there exists a point $w_n$ in the segment $[z_n,\xi_n)\subset D$
such that $$\left|k_D(w_n,p)+\log\delta_D(w_n)+\log\frac{|\Omega_{\xi_n}(p)|}{2}\right|<\varepsilon$$
By Lemma \ref{lemmaestimate} we have that, for  $n$ large enough, 
$$
k_D(z_n,w_n)\leq\log\frac{\delta_D(z_n)}{\delta_D(w_n)}+\varepsilon.
$$
Moreover,
\begin{align*}
k_D(z_n,p)+\log\delta_D(z_n)&\geq k_D(w_n,p)-k_D(z_n,w_n)+\log\delta_D(z_n)
\\&\geq k_D(w_n,p)-\log\frac{\delta_D(z_n)}{\delta_D(w_n)}+\log\delta_D(z_n)-\varepsilon
\\&=k_D(w_n,p)+\log\delta_D(w_n)-\varepsilon
\\&\geq-\log\frac{|\Omega_{\xi_n}(p)|}{2}-2\varepsilon.\end{align*}
 Since $\xi_n\to\xi$ we have, by Proposition \ref{contpole},  $\Omega_{\xi_n}(p)\to\Omega_\xi(p)$, and thus
$$\liminf_{n\to+\infty}k_D(z_n,p)+\log\delta_D(z_n)\geq-\log\frac{|\Omega_{\xi}(p)|}{2}.$$
\textbf{Upper bound:} 
Assume first that $(z_n)$ is a   sequence  $K'$-converging to $\xi$. Let $\varphi\colon\D\to D$ be a complex geodesic  with $\varphi(0)=p$ and with endpoint $\xi$. Then by definition  there exists  a sequence $(\zeta_n)$ in $\D$ converging to 1 non-tangentially such that $k_D(z_n,\varphi(\zeta_n))\to0$. By \cite[Proposition 2.4]{Mercer}, for all $z,w\in D$ we have
$$\left|\log\frac{\delta_D(z)}{\delta_D(w)}\right|\leq k_D(z,w),$$
thus
\begin{align*}
	k_D(z_n,p)+\log\delta_D(z_n)&\leq k_D(\varphi(\zeta_n),p)+k_D(z_n,\varphi(\zeta_n))+\log\delta_D(z_n)
	\\&\leq k_D(\varphi(\zeta_n),p)+\log \delta_D(\varphi(\zeta_n))+2k_D(z_n,\varphi(\zeta_n)).
	\end{align*}
Hence it is enough to prove the result for the sequence $(\varphi(\zeta_n))$, and this follows from  Lemma~\ref{lemmaexclaim}:
\begin{align*}
	\limsup_{n\to+\infty}[k_D(\varphi(\zeta_n),p)+\log\delta_D(\varphi(\zeta_n))]&=	\limsup_{n\to+\infty} \left[\log\frac{1+|\zeta_n|}{1-|\zeta_n|}+\log\delta_D(\varphi(\zeta_n))\right]\\&\leq\log(2\varphi_N'(1))\\&=-\log\frac{|\Omega_\xi(p)|}{2}.
\end{align*}  Assume now that  $(z_n)$ is a sequence converging to $\xi$, and let $\xi_n:=\pi(z_n)\in\partial D$ be the closest boundary point.
Set $t_n:=\max\{\|\xi-\xi_n\|,\delta_D(z_n)\}$, and define a sequence   as
$ w_n:=\xi_n-t_nn_{\xi_n}.$
If $n$ is large enough, $w_n$ is contained in $D$ and satisfies $\delta_D(w_n)=t_n$ and $\pi(w_n)=\xi_n$. It then follows that 
$$\|w_n-\xi\|\leq 2\delta_D(w_n),$$ and thus $(w_n)$ converges to $\xi$ non-tangentially (in the sense of cones). By \cite[Theorem~7.9]{Arosio-Fiacchi}, this implies that $(w_n)$ $K'$-converges to $\xi$.
Fix $\varepsilon>0$.
Using again Lemma \ref{lemmaestimate} we have that, for $n$ large enough,
\begin{align*}
	k_D(z_n,p)+\log\delta_D(z_n)&\leq k_D(w_n,p)+k_D(z_n,w_n)+\log\delta_D(z_n)
	\\&\leq k_D(w_n,p)+\log\frac{\delta_D(w_n)}{\delta_D(z_n)}+\log\delta_D(z_n)+\varepsilon
	\\&=k_D(w_n,p)+\log\delta_D(w_n)+\varepsilon.\end{align*}
By the first part of the proof $$\limsup_{n\to+\infty}[k_D(w_n,p)+\log\delta_D(w_n)]\leq-\log\frac{|\Omega_{\xi}(p)|}{2},$$ and thus
$$\limsup_{n\to+\infty}[k_D(z_n,p)+\log\delta_D(z_n)]\leq-\log\frac{|\Omega_{\xi}(p)|}{2}.$$
\endproof
The following corollary is immediate.
\begin{corollary}\label{kDvskD'}
	Let $D\subset\C^n$ be a $\C$-proper convex domain and let  $\xi\in\partial D$ be a boundary point of finite type. Let $U$ be a neighborhood of $\xi$ such that $D':=D\cap U$ is convex. Then
	$$\lim_{z\to\xi}k_{D'}(z,p)-k_{D}(z,p)=\log\frac{\Omega^D_\xi(p)}{\Omega^{D'}_\xi(p)}.$$
\end{corollary}

The next result is a generalization of \cite[Proposition 4.11]{BST} to $\C$-proper convex domains of finite type. 
With Corollary  \ref{kDvskD'} at hand the proof is remarkably simpler.
\begin{corollary}
	Let $D\subset\C^n$ be a $\C$-proper convex domain and let  $\xi\in\partial D$ be a boundary point of finite type. Let $U$ be a neighborhood of $\xi$ such that $D':=D\cap U$ is convex. Then
	$$\lim_{z\to\xi}\frac{\Omega_\xi^D(z)}{\Omega_\xi^{D'}(z)}=1.$$
\end{corollary}
\proof
By \cite[Proposition 6]{NO} we have
$$\lim_{z,w\to\xi}[k_{D'}(z,w)-k_{D}(z,w)]=0,$$
that is,  for all $\varepsilon>0$ we can find a neighborhood $V$ of $\xi$ such that for all $z,w\in D'\cap V$
$$|k_{D'}(z,w)-k_{D}(z,w)|\leq \varepsilon.$$
  Letting $w\to\xi$ we have, by Corollary \ref{kDvskD'}, 
$$\left|\log\frac{\Omega^D_\xi(z)}{\Omega^{D'}_\xi(z)}\right|\leq \varepsilon,$$
for all $z\in D'\cap V$.
Thus
$$\lim_{z\to\xi}\frac{\Omega_\xi^D(z)}{\Omega_\xi^{D'}(z)}=1.$$
\endproof

\section{$\Omega_\xi$ and the pluriharmonic  measure}\label{sectDem}
In this section we show that putting together Theorem \ref{Kobestimate} with  the theory of pluriharmonic measures    developed by Demailly in \cite{Demailly}, we obtain a reproducing formula for plurisubharmonic functions. The proof in \cite{BPT} of the same result for strongly convex domains does not work in our setting since it is  based on the regularity of the Green function.
Let $D\subset \C^n$ be a bounded convex domain of finite type. Let $\rho$ be a $C^2$-smooth defining function of $D$, and let $\omega_D$ be the $(2n-1)$-form defined by 
\begin{equation}\label{Eq:Levi-form}
\omega_{\partial D}:=\frac{(dd^c\rho)^{n-1}\wedge d^c\rho}{\|d\rho\|^n}|_{\partial D},
\end{equation}
which is semipositive and independent on the defining function $\rho$. 

\begin{remark}
Let $\mathcal{L}_\rho$ denote the Levi form of $\rho$. It is proved in \cite[Remark 9.1]{BST} that
$$\omega_{\partial D}=4^{n-1}(n-1)!\frac{\mbox{det}(\mathcal{L}_\rho)}{\|d\rho\|^{n-1}}d\mbox{Vol}_{\partial D}.$$ In particular $\omega_{\partial D}$ is supported on the subset of the strongly pseudoconvex points of the boundary.
\end{remark}
Let $\varphi\colon D\to [-\infty,0)$ be a continuous plurisubharmonic exhaustion function with finite  total Monge--Amp\`ere mass, that is
$$\int_D (dd^c\varphi)^n<+\infty.$$ Demailly \cite{Demailly} associates to every such $\varphi$ a measure $\mu_\varphi$ on $\partial D$ with mass $\int_D (dd^c\varphi)^n$.   When $\varphi=G_z/2\pi$ with $z\in D$, the probability measure $\mu_\varphi$ is denoted $\mu_z$ and  is called the {\it pluriharmonic  measure} at $z$.
If $\varphi$ is a defining function of $D$ which is $C^2$-smooth near the boundary, then by Demailly's construction one has
$$d\mu_\varphi=((dd^c\varphi)^{n-1}\wedge d^c\varphi)|_{\partial D}.$$ In particular, if $\|d\varphi\|=1 $ on $\partial D$, it follows that $d\mu_\varphi=\omega_D$.

Notice that the signed distance function $r:=-\delta_D\colon \overline{D}\to (-\infty,0]$ is a convex defining  function (in particular it is plurisubharmonic), which is  $C^2$-smooth  near the boundary and satisfies $\|dr\|=1$ on $\partial D$, and thus $d \mu_r=\omega_D$.
\begin{proposition}\label{limitratio}
For all $\xi\in\partial D$ we have
$$\lim_{w\to\xi}\frac{G_z(w)}{r(w)}=|\Omega_\xi(z)|.$$
\end{proposition}
\proof
We have $$\frac{G_z(w)}{r(w)}=\frac{G_z(w)}{-2e^{-k_D(z,w)}}\frac{2e^{-k_D(z,w)}}{\delta_D(w)},$$
and $\frac{G_z(w)}{-2e^{-k_D(z,w)}}\stackrel{w\to\xi}\longrightarrow1$, while, by Proposition \ref{Kobestimate},
$\frac{2e^{-k_D(z,w)}}{\delta_D(w)}\stackrel{w\to\xi}\longrightarrow |\Omega_\xi(z)|.$
\endproof

\begin{theorem}
Let $D\subset\C^n$ be a bounded convex domain of finite type. The pluriharmonic  measure $\mu_z$ satisfies $d\mu_z(\xi)=\frac{1}{(2\pi)^n} |\Omega_\xi(z)|^n \omega_{\partial D}(\xi)$. In particular, if $F\in\psh(D)\cap C^0(\overline{D})$, then for all $z\in D$  
$$F(z)=\frac{1}{(2\pi)^n}\int_{\xi\in\partial D}|\Omega_{\xi}(z)|^nF(\xi)\omega_{\partial D}(\xi)-\frac{1}{(2\pi)^n}\int_{w\in D}|G_z(w)|dd^cF(w)\wedge(dd^cG_z)^{n-1}(w).$$
 If $F$ is pluriharmonic then
$$F(z)=\frac{1}{(2\pi)^n}\int_{\xi\in\partial D}|\Omega_{\xi}(z)|^nF(\xi)\omega_{\partial D}(\xi).$$
\end{theorem}
\proof
It follows from \cite[Théorème 3.8]{Demailly} and Proposition \ref{limitratio} that for all $z\in D$,
$d\mu_{G_z}(\xi)=|\Omega_\xi(z)|^nd\mu_r(\xi),$
that is $$d\mu_z(\xi)=\frac{1}{(2\pi)^n} |\Omega_\xi(z)|^n \omega_{\partial D}(\xi).$$
\endproof
\begin{remark}
As a consequence, in agreement with \cite[Th\'eor\`eme 6.1]{Demailly}, the pluriharmonic  measure $\mu_z$ is supported on the subset of strongly pseudoconvex points of the boundary.

\end{remark}

\section{A new characterization of the dilation}\label{sectiondilation}

In this section we apply our results to the study of boundary behaviour of holomorphic mappings. We will consider only the case of  codomains which are of (global) finite type, but similar (more technical) results can be obtained assuming the finite type condition only locally around a boundary point of the codomain. We begin by recalling the relevant definition and results.
\begin{definition}[Dilation]
Let $D\subset\C^n$ be a $\C$-proper convex domain and let  $\xi\in\partial D$ be a point of locally finite type.
Let $D'\subset \C^m$ be a bounded convex domain of finite type. Let $f\colon D\to D'$ be a holomorphic map, and let $p\in D, p'\in D'.$
The {\it dilation} of $f$ at $\xi$, with base-points $p,p'$, is the number $\lambda_{\xi,p,p'}\in (0,+\infty]$ defined by 
 $$\log\lambda_{\xi,p,p'}:=\liminf_{z\to\xi} [k_D(z,p)-k_{D'}(f(z),p')].$$
 The point $\xi$ is a {\it regular contact point} if $\lambda_{\xi,p,p'}<+\infty$.
 \end{definition}

 The definition of regular contact point is independent of the chosen base-points. Moreover,
 if $\xi$ is a regular contact point for $f\colon D\to D'$, then there exists  a point $\eta\in \partial D'$ such that 
 $K\textrm{-}\lim_{z\to\xi}f(z)=\eta$. (For proofs, see \cite[Section 10]{Arosio-Fiacchi}). 
Clearly the dilation   depends on the chosen base-point. However, the pluricomplex Poisson kernel can be used to normalize the dilation, hence making it independent of the base-points.
 \begin{definition}[Normalized dilation]
 Let $D\subset\C^n$ be a $\C$-proper convex domain and let  $\xi\in\partial D$ be a point of locally finite type.
Let $D'\subset \C^m$ be a bounded convex domain of finite type. Let $f\colon D\to D'$ be a holomorphic map, and assume that $\xi$ is a regular contact point with
 $K\textrm{-}\lim_{z\to\xi}f(z)=\eta\in \partial D'$.
The {\it normalized dilation} $\alpha_\xi$ of $f$ at $\xi$ is 
$$\alpha_\xi:=\lambda_{\xi,p,p'}\frac{\Omega^D_\xi(p)}{\Omega^{D'}_{\eta}(p')}\in (0,+\infty),$$ where $p\in D,p'\in D'$.
 \end{definition}
 The number $\alpha_\xi$ does not depend on $p,p'$, see \cite[Section 11]{Arosio-Fiacchi}.
  \begin{remark}
  Let $f\colon \D\to \D$ be holomorphic. Then by the classical Julia--Wolff--Carath\'eodory theorem (see e.g. \cite[Theorem 1.7.3]{booksemigroup})
  the point    $\xi\in \partial \D$ is a regular contact point iff $f$ has finite angular derivative at $\xi$ and $f(\xi)\in \partial \D$. 
Moreover if this is the case we have 
\begin{equation}\label{diskcase}|f'(\xi)|=\alpha_\xi=\lambda_{\xi,0,0}.
\end{equation}
 \end{remark}

The following result summarizes the principal known results concerning the  dilation. 
\begin{theorem}\label{J-JWC}Let $D\subset\C^n$ be a $\C$-proper convex domain and let  $\xi\in\partial D$ be a point of locally finite type. Let $D'\subset \C^{m}$ be a bounded convex domain of finite type. Let $f\colon D\to D'$ be a holomorphic map, and assume that $\xi$ is a regular contact point with
 $K\textrm{-}\lim_{z\to\xi}f(z)=\eta\in \partial D'$. Then
\begin{enumerate}
\item[(MJ)]$\sup_{z\in D} [h^{D'}_{\eta,p'}(f(z))-h^D_{\xi,p}(z)]=\log\lambda_{\xi,p,p'},$
\item[(PJ)]$\sup_{z\in D} \frac{\Omega^D_{\xi}(z)}{\Omega^{D'}_\eta(f(z))}=\alpha_\xi,$
\item[(JWC)] $K'\textrm{-}\lim_{z\to\xi}\langle df_z(n_\xi),n_\eta\rangle=\alpha_\xi.$
\end{enumerate}
\end{theorem}
\proof
(MJ) has been proven in \cite[Proposition 10.10]{Arosio-Fiacchi} and (JWC) in \cite[Theorem~12.1]{Arosio-Fiacchi}. Finally, (PJ) follows from (MJ) and Corollary \ref{omegahorofunction}.

\endproof
\begin{remark}
(MJ) and (PJ) are generalizations of the classical Julia Lemma, while (JWC) is a generalization of the classical Julia--Wolff--Carath\'eodory theorem. For more references and context, we refer to \cite{Arosio-Fiacchi}.
\end{remark} 
We now prove our first result of this section. 	\begin{proposition}\label{K'dilation}
	Let $D\subset\C^n$ be a $\C$-proper convex domain and let  $\xi\in\partial D$ be a point of locally finite type. Let $D'\subset \C^{m}$ be a bounded convex domain of finite type. Let $f\colon D\to D'$ be a holomorphic map, and assume that $\xi$ is a regular contact point with
 $K\textrm{-}\lim_{z\to\xi}f(z)=\eta\in \partial D'$. Let $p\in D, p'\in D'$. Then
	\begin{itemize}
		\item[(i)] $ K'\textrm{-}\lim_{z\to\xi}[k_D(z,p)-k_{D'}(f(z),p')]=\log\lambda_{\xi,p,p'},$
		\item[(ii)] $K'\textrm{-}\lim_{z\to\xi}[h^{D'}_{\eta,p'}(f(z))-h^D_{\xi,p}(z)]=\log\lambda_{\xi,p,p'},$
		\item[(iii)] $ K'\textrm{-}\lim_{z\to\xi}\frac{\Omega^D_{\xi}(z)}{\Omega^{D'}_\eta(f(z))}=\alpha_\xi.$
	\end{itemize}
\end{proposition}

\proof
Let $\varphi\colon\D\to D$ be a complex geodesic with $\varphi(0)=p$ and with endpoint $\xi$, and let $\psi\colon\D\to D'$ be a complex geodesic with $\psi(0)=p'$ and with endpoint $\eta$. Let $\tilde\rho\colon D'\to\D$ be the left inverse of $\psi$.
By \cite[Proposition 10.10]{Arosio-Fiacchi} we have $\lambda_{\eta,p',0}(\tilde\rho)=1$. Moreover  $\lambda_{1,0,p}(\varphi)=1$, and  
\begin{equation}\label{eqjul}h^{D}_{\xi,p}(\varphi(\zeta))=h^{\D}_{1,0}(\zeta),\quad \forall \zeta\in \D.
\end{equation}
 By the chain rule for the  dilation (\cite[Corollary 10.13]{Arosio-Fiacchi})  it follows that
\begin{equation}\label{normdilcalculation}
\lambda_{1,0,0}(\tilde\rho\circ f\circ \varphi)=\lambda_{\eta,p',0}(\tilde\rho)\cdot\lambda_{\xi,p,p'}(f)\cdot \lambda_{1,0,p}(\varphi)=\lambda_{\xi,p,p'}(f).\end{equation}

Let $(z_n)$ be a sequence  in $D$ that $K'$-converges to $\xi$. 
By definition there exists a sequence $(\zeta_n)$ in $\D$ converging to $1$ non-tangentially such that $k_D(z_n,\varphi(\zeta_n))\to0$.

(i) By definition 
$$\liminf_{n\to+\infty}[k_{D}(z_n,p)-k_{D'}(f(z_n),p')]\geq\log\lambda_{\xi,p,p'}.$$
On the other hand, by (\refeq{diskcase}),
\begin{align*}&\limsup_{n\to+\infty}[k_{D}(z_n,p)-k_{D'}(f(z_n),p')]\leq\limsup_{n\to+\infty}[k_{D}(\varphi(\zeta_n),p)-k_{D'}(f(\varphi(\zeta_n)),p')+2k_D(z_n,\varphi(\zeta_n))]
	\\\leq&\limsup_{n\to+\infty}[k_{\D}(\zeta_n,0)-k_{D'}(f(\varphi(\zeta_n)),p')]
	\leq\limsup_{n\to+\infty}[k_{\D}(\zeta_n,0)-k_{\D}(\tilde\rho(f(\varphi(\zeta_n))),0)]
	=\log\lambda_{1,0,0}(\tilde\rho\circ f\circ \varphi),
	\end{align*}
	and thus the result follows from \eqref{normdilcalculation}.

(ii) By (MJ) in Theorem \ref{J-JWC} we have, for all $z\in D$,
$$h^{D'}_{\eta,p'}(f(z))-h^D_{\xi,p}(z)\leq\log\lambda_{\xi,p,p'}.$$
Moreover, by \eqref{eqjul} we have, for $n\geq 0$,
\begin{align*}h^{D'}_{\eta,p'}(f(z_n))-h^D_{\xi,p}(z_n)&\geq h^{D'}_{\eta,p'}(f(\varphi(\zeta_n)))-h^{D}_{\xi,p}(\varphi(\zeta_n))-2k_D(z_n,\varphi(\zeta_n))\\&\geq h^\D_{1,0}(\tilde\rho(f(\varphi(\zeta_n))))-h^{\D}_{1,0}(\zeta_n)-2k_D(z_n,\varphi(\zeta_n)),\end{align*}
so by (\refeq{ntpoidisk}) and (\refeq{diskcase}) (recalling that $h_{1,0}^\D=-\log|\Omega^\D_1|$)
$$\limsup_{n\to+\infty}[h^{D'}_{\eta,p'}(f(z_n))-h^D_{\xi,p}(z_n)]\geq \limsup_{n\to+\infty}[h^\D_{1,0}(\tilde\rho(f(\varphi(\zeta_n))))-h^{\D}_{1,0}(\zeta_n)]=\log\lambda_{1,0,0}(\tilde\rho\circ f\circ \varphi),$$
and thus the result follows from \eqref{normdilcalculation}.

(iii) By the Poisson-horofunction formula (Proposition \ref{omegahorofunction}) we have, for all $n\geq 0$,	$$\frac{\Omega_\xi^D(z_n)}{\Omega_\eta^{D'}(f(z_n))}=\frac{\Omega_\xi^D(p)}{\Omega_\eta^{D'}(p')}\exp(h^{D'}_{\eta,p'}(f(z_n))-h^{D}_{\xi,p}(z_n)),$$
so the result follows from (ii).
\endproof

It is remarked in \cite{AbateSymp} that, assuming $D,D'$ bounded strongly convex, the point     $\xi\in\partial D$ is a regular contact point if and only if 
\begin{equation}\label{distdil}\liminf_{z\to\xi}\frac{\delta_{D'}(f(z))}{\delta_{D}(z)}<+\infty.
\end{equation}
The exact value of this liminf was previously unknown. Thanks Theorem \ref{Kobestimate} we are able to show that this value is the normalized dilation.
	\begin{theorem}\label{deltaratio}
	Let $D\subset\C^n$ be a $\C$-proper convex domain and let  $\xi\in\partial D$ be a point of locally finite type. Let $D'\subset \C^{m}$ be a bounded convex domain of finite type. Let $f\colon D\to D'$ be a holomorphic map. Then  $\xi$ is a regular contact point if and only if \eqref{distdil}  holds, and
	$$\liminf_{z\to\xi}\frac{\delta_{D'}(f(z))}{\delta_{D}(z)}=K'\textrm{-}\lim_{z\to\xi}\frac{\delta_{D'}(f(z))}{\delta_{D}(z)}=\alpha_\xi.$$
	\end{theorem}
\proof
The equivalence between the regularity of $\xi$ and \eqref{distdil} it has been proven in \cite[Proposition 10.15]{Arosio-Fiacchi}.
Let $(z_ {n})$ be a sequence that realizes the limit inferior  \eqref{distdil}. Clearly we can suppose that $f(z_{n})$ converges to a boundary point $\eta\in\partial D'$.
By  \cite[Lemma 3.8]{Arosio-Fiacchi}, the sequence $(k_{D}(w_n,p)-k_{D}(f(w_n),p'))$ is bounded. It then follows by \cite[Lemma 10.4]{Arosio-Fiacchi} that the $K$-limit of $f$ at $\xi$ is $\eta$.
 By Theorem \ref{Kobestimate} we have
$$\lim_{n\to+\infty}\frac{\delta_{D'}(f(z_{n}))}{\delta_{D}(z_{n})}=\frac{\Omega^{D}_\xi(p)}{\Omega^{D'}_\eta(p')}\lim_{n\to+\infty}[\exp(k_{D}(z_{n},p)-k_{D'}(f(z_{n}),p'))]\geq\lambda_{\xi,p,p'}\frac{\Omega^{D}_\xi(p)}{\Omega^{D'}_\eta(p')}=\alpha_\xi.$$

Finally, the previous computations shows, together with (i) in Proposition \ref{K'dilation}, that 
$K'\textrm{-}\lim_{z\to\xi}\frac{\delta_{D'}(f(z))}{\delta_{D}(z)}=\alpha_\xi,$ and thus $\liminf_{z\to\xi}\frac{\delta_{D'}(f(z))}{\delta_{D}(z)}=\alpha_\xi.$
\endproof

\begin{remark}
The results in Proposition \ref{K'dilation} and Theorem \ref{deltaratio} do not hold if we replace $K'$-limits with $K$-limits:
take for example $f\colon\B^2\to\D$ given by $f(z_0,z_1)=z_0$ and for all $\lambda\in\D$ $\gamma_\lambda:[0,1)\to\B^2$ given by $\gamma_\lambda(t)=(t,\lambda\sqrt{1-t^2})$.
Clearly $\alpha_{e_1}=\lambda_{e_1,0,0}=1$ but
$$\lim_{t\to1^-}k_{\B^2}(\gamma_\lambda(t),0)-k_{\D}(f(\gamma_\lambda(t)),0)=-\log(1-|\lambda|^2),$$
$$\lim_{t\to1^-}h^{\D}_{1,0}(f(\gamma_\lambda(t)))-h^{\B^2}_{e_1,0}(\gamma_\lambda(t))=\log(1-|\lambda|^2),$$
$$\lim_{t\to1^-}\frac{\Omega^{\B^2}_{e_1}(\gamma_\lambda(t))}{\Omega^{\D}_{1}(f(\gamma_\lambda(t)))}=1-|\lambda|^2,$$
and
$$\lim_{t\to1^-}\frac{\delta_\D(f(\gamma_\lambda(t)))}{\delta_{\B^2}(\gamma_\lambda(t))}=\frac{1}{1-|\lambda|^2}.$$
\end{remark}

 We end this section studying the holomorphic maps which preserve the pluricomplex  Poisson kernel.
\begin{proposition}
	Let $D\subset\C^n$ be a $\C$-proper convex domain and let  $\xi\in\partial D$ be a point of locally finite type. Let $D'\subset \C^{m}$ be a bounded convex domain of finite type. Let $f\colon D\to D'$ be a holomorphic map. Then the following are equivalent.
	\begin{enumerate}
		\item There exists $\eta\in \partial D'$ and $\alpha>0$ such that 
		$$\alpha\Omega_\eta^{D'}(f(z))=\Omega_\xi^D(z), \quad \forall\, z\in D;$$
		\item for every complex geodesic $\varphi$ with endpoint $\xi$ we have that  $f\circ \varphi$ is a complex geodesic of $D'$;
		\item for  every $z\in D$ there exists a complex geodesic $\varphi$ with $\varphi(0)=z$ and  endpoint $\xi$ such that  $f\circ \varphi$ is a complex geodesic of $D'$.
	\end{enumerate}
\end{proposition}
\begin{proof}
	(1) $\Rightarrow$ (2). The point $\xi$ is a regular contact point with $K$-limits $\eta$ by \cite[Proposition 12.13]{Arosio-Fiacchi}. By (PJ) in Theorem \ref{J-JWC} $\alpha$ is the normalized dilation of $f$ at $\xi$.
	Now the result follows as in \cite[Theorem 2.4]{BCD}.
	(2) $\Rightarrow$ (3) is trivial.
	(3) $\Rightarrow$ (1). Fix $z\in D$. The point $\xi$ is a regular contact point with $K$-limit $\eta$. By \cite[Proposition 10.10]{Arosio-Fiacchi}  
	$$\log\lambda_{\xi,z,f(z)}=\lim_{t\to 1-}[k_D(\varphi(t),z)-k_{D'}(f(\varphi(t)),f(z))]=0.$$
By definition of normalized dilation we have
	$\alpha_\xi=\Omega_\xi^D(z)/\Omega_\eta^{D'}(f(z)).$ Since this is true for all $z\in D$, we have (1).
\end{proof}
\begin{remark}
	If $D,D'$ are bounded strongly convex domains, then any map satisfying one of the equivalent conditions of the previous proposition is actually a biholomorphism by \cite{BKZ}. This is no longer the case if $D,D'$ are convex domains of finite type, as the following example shows.
\end{remark}

\begin{example}
Let $m:=(m_1,\cdots,m_{n-1})$ be an array of even integers, and let $\mathbb{E}_m$ be as in Example \ref{ellipsoid}.
If $\xi=(1,0)\in\C\times\C^{n-1}$, then the map $f\colon\mathbb{E}_m\to \B^n$ given by
$$f(z)=(z_0,z_1^{m_1/2},\cdots,z_{n-1}^{m_{n-1}/2})$$
has the property that 
$\Omega^{\B^n}_\xi(f(z))=\Omega^{\mathbb{E}_m}_\xi(z)$ for all $z\in\mathbb{E}_m.$
\end{example}

\section{A counterexample  in the non-convex strongly pseudoconvex case}\label{controesempioanello}
In this section we show that the Poisson-horofunction formula \eqref{poisson-horofunction-intro} fails in general in strongly pseudoconvex domains.
Let $D\subset\C^n$ be a smooth strongly pseudoconvex domain and $\xi\in\partial D$.
According to \cite{BST}, the pluricomplex Poisson kernel $\Omega_\xi$ is the greatest element of the family $\mathscr{G}^D_\xi$ given by
	\begin{equation}\label{MAmaxSTRPSEUDO}\begin{cases*}
		u\mbox{ is plurisubharmonic in }D\\
		u<0 \ \mbox{ in }D\\
		\limsup_{t\to1^-}u(\gamma(t))(1-t)\leq-\Re\frac{2}{\langle\gamma'(1),n_\xi\rangle}, \ \ \forall \gamma\in\Gamma_\xi,
	\end{cases*}
\end{equation}
where  $\Gamma_\xi$ is the family of $C^1$-smooth curves $\gamma\colon[0,1]\to D\cup\{\xi\}$ such that $\gamma(t)\in D$ for $t\in[0,1)$, $\gamma(1)=\xi$ and $\gamma'(1)\notin T_\xi\partial D$.
The function $\Omega_\xi$ is maximal plurisubhamonic \cite[Proposition 2.8]{BST}.

\begin{lemma}\label{hvsomegaannulus}
Let $r\in(0,1)$ and consider the annulus $A_r:=\{z\in\C:r<|z|<1\}$. Then for all $\xi\in\partial A_r$ and  $p\in (r,1)$ the limit 
\begin{equation}\label{horoannulus}
h^{A_r}_{\xi,p}(z):=\lim_{w\to\xi}[k_{A_r}(z,w)-k_{A_r}(w,p)]
\end{equation}
exists uniformly on compact sets of $A_r$. Moreover, the function $u\colon A_r\to\R$ given by $$u(z):=-\exp(-h^{A_r}_{\xi,p}(z))$$ is not harmonic.
\end{lemma}
\proof
The existence of the limit (\refeq{horoannulus}) is proved in \cite[Proposition 3.3.12]{Abatebooknew} (note that it can also be deduced from \cite[Theorem 3.11, Theorem 4.3]{AFGG}, since on the annulus the squeezing function tends to 1 at the boundary).

Up to a rotation and involution, we can suppose that $\xi=1\in\partial A_r$. Let $S_r:=\{z\in\C: \ln r<\Re z<0\}$ and let  $\Pi\colon S_r\to A_r$ 
be the covering map  $\Pi(z)=\exp(z)$. Notice that 
$ \lim_{z\to0}\Pi(z)=1$.
Consider the fundamental domain 
$$Q:=\{z\in\C: \log r<\Re z<0, -\pi\leq|\Im z|<\pi\}\subset S_r.$$
Recall that the Kobayashi distance of $A_r$ is given, for all $z,w\in A_r$, by
$$k_{A_r}(z,w)=\inf_{\substack{\tilde z\in\Pi^{-1}(z)\\ \tilde w\in\Pi^{-1}(w)}}k_{S_r}(\tilde z,\tilde w).$$
Let $q:=\log p$.  It easy to see that if $z\in Q$ and $t\in S_r\cap \R$ we have
$$k_{S_r}(z,t)=k_{A_r}(\Pi(z),\Pi(t)), \ \ k_{S_r}(t,q)=k_{A_r}(\Pi(t),p)$$
which means that for all $z\in Q$
$$h^{A_r}_{1,p}(\Pi(z)):=\lim_{t\to 0^-}k_{A_r}(\Pi(z),\Pi(t))-k_{A_r}(\Pi(t),p)=\lim_{t\to 0^-}k_{S_r}(z,t)-k_{S_r}(t,q)=h^{S_r}_{0,q}(z).$$
Let $\nu\colon S_r\to \R$ be the function $\nu(z)=-\exp(-h^{S_r}_{0,q}(z))$. It is a harmonic function, indeed up to a multiplicative factor it is the Poisson kernel with pole at $0$ in $S_r$.
Let $u\colon A_r\to\R$ given by $u(z)=-\exp(-h^{A_r}_{1,p}(z))$.  Since $\Pi$ is a local biholomorphism, $u$ is harmonic if and only if $v:=u\circ \Pi$ is harmonic.
Assume by contradiction that $v$ is harmonic on $S_r$. Since it coincides with $\nu$ on $Q$,  $v$ and $\nu$ must coincide on all of $S_r$. However, $v$ is invariant by $F(z)=z+2i\pi$, while $\nu$ is clearly not.\endproof

It follows from the previous result that the  Poisson-horofunction formula \eqref{poisson-horofunction-intro} does not hold in the annulus. We now give a higher dimensional example.
Let $r\in(0,1)$ and consider the domain
$$D:=\left\{z\in\C^2: r<|z_1|<1, |z_2|^2< \sin\left(\frac{\pi\log|z_1|}{\log r}\right)\right\}.$$
By  \cite{defabr}
$D$ is a bounded strongly pseudoconvex domain that is a quotient of the  ball.
The existence of the horofunction in strongly pseudoconvex domains was proved in \cite[Corollary 4.5]{AFGG}.
\begin{proposition}
Let $D\subset \C^2$ be as above.
Let $\xi=(1,0)\in\partial D$ and $p\in D$, then 
$h^D_{\xi,p}$ and $\Omega^D_\xi$ have different sublevel sets. In particular the Poisson-horofunction formula \eqref{poisson-horofunction-intro} does not hold.
\end{proposition}
\proof
We can suppose that $p=(p_1,0)\in D$. By contradiction, suppose that $h^D_{\xi,p}$ and $\Omega^D_\xi$ have the same sublevel sets, which means that there exists an increasing homeomorphism $Y\colon \R\to (-\infty,0)$ such that for all $z\in D$
$$\Omega^D_\xi(z)=Y(h^D_{\xi,p}(z)).$$
Since $D$ is a smooth strongly pseudoconvex domain by \cite{BFW} there exists a complex geodesic $\varphi\colon\D\to D$ with endpoint $\xi$. Then, for all $\zeta\in\D$,
$$h^D_{\xi,p}(\varphi(\zeta))=h^D_{\xi,\varphi(0)}(\varphi(\zeta))+h^D_{\xi,p}(\varphi(0))=h^\D_{1,0}(\zeta)+h^D_{\xi,p}(\varphi(0))=-\log|\Omega^\D_1(\zeta)|+h^D_{\xi,p}(\varphi(0)).$$
By Corollary \ref{forPoigeo} we have $\Omega_\xi^D(\varphi(\zeta))=\frac{\Omega_1^\D(\zeta)}{\varphi'_N(1)}$, hence there exists $c_p>0$ such that $Y(t)=-c_p\exp(-t)$, that is, for all $z\in D$
\begin{equation}\label{eqann1}
\Omega^D_\xi(z)=-c_p\exp(-h^D_{\xi,p}(z)).
\end{equation}
Now, let $\rho\colon D\to D$ given by $\rho(z_1,z_2)=(z_1,0)$. Notice that it is a holomorphic retraction of $D$ and $M:=\rho(D)=A_r\times\{0\}$ is a holomorphic retract which is biholomorphic to the annulus $A_r$. This implies that $k_{D}|_{M\times M}=k_{M}$, and thus for all $z_1\in A_r$ we have
\begin{equation}\label{eqann2}h^D_{\xi,p}(z_1,0)=h_{1,p_1}^{A_r}(z_1).\end{equation}
Now we will show that, for all $z_1\in A_r$,
\begin{equation}\label{eqann3}\Omega^D_\xi(z_1,0)=\Omega^{A_r}_1(z_1).\end{equation}
First of all, the function $z\in D\mapsto\Omega_1^{A_r}(z_1)$ is in the family $\mathscr{G}^D_\xi$, so by maximality $\Omega_1^{A_r}(z_1)\leq \Omega^D_\xi(z_1,0)$. We obtain the other inequality noticing that the function $z\in A_r\to \Omega^D_\xi(z,0)$ is in the family $\mathscr{G}^{A_r}_1$, so $\Omega^D_\xi(z_1,0)\leq \Omega_1^{A_r}(z_1)$.

So combining (\refeq{eqann1}), (\refeq{eqann2}) and (\refeq{eqann3}), we have, for all $z_1\in A_r$, $$\Omega^{A_r}_1(z_1)=-c_p\exp(-h_{1,p_1}^{A_r}(z_1)),$$
but this contradicts Lemma \ref{hvsomegaannulus}.
\endproof


\begin{thebibliography}{88} 
	
\bibitem[Aba89]{Abatebook} M. Abate, {\sl Iteration theory of holomorphic maps on taut manifolds}, Research and Lecture Notes in Mathematics. Complex Analysis and Geometry. Mediterranean Press, Rende (1989)

\bibitem[Aba91]{AbateSymp} M. Abate, {\sl Angular derivatives in strongly pseudoconvex domains. Several complex variables and complex geometry}, Part 2 (Santa Cruz, CA, 1989), 23-40, Proc. Sympos. Pure Math., 52, Part 2, Amer. Math. Soc., Providence, RI, (1991).

\bibitem[Aba22]{Abatebooknew} M. Abate, {\sl Holomorphic dynamics on hyperbolic Riemann surfaces}, De Gruyter, Berlin, (2022).

\bibitem[AT02]{AbTau} M. Abate, R. Tauraso, {\sl The Lindel\"of principle and angular derivatives in convex domains of finite type}, J. Aust. Math. Soc. {\bf 73} (2002), no. 2, 221--250.

\bibitem[AF25a]{approaching} L. Arosio, M. Fiacchi, {\sl On the approaching geodesics property}, Boll. Un. Mat. Ital., {\bf 18} (2025), no.1,   3--16.

\bibitem[AF25b]{Arosio-Fiacchi}
L. Arosio, M. Fiacchi, {\sl The Julia--Wolff--Carath\'eodory theorem in convex domains of finite type}, J. Reine Angew. Math.,  {\bf 827} (2025), 59–114.
	
\bibitem[AFGG22]{AFGG} L. Arosio, M. Fiacchi, S. Gontard, L. Guerini, {\sl The horofunction boundary of a Gromov hyperbolic space}, Math. Ann., {\bf 388} (2022), 1163--1204.

\bibitem[BB00]{BaBo} Z. M. Balogh, M. Bonk, {\sl Gromov hyperbolicity and the Kobayashi metric on strictly pseudoconvex domains}, Comment. Math. Helv. {\bf 75} (2000), 504--533.

\bibitem[Bar80]{Barth} T. J. Barth, {\sl Convex domains and Kobayashi hyperbolicity}, Proc. Amer. Math. Soc. {\bf 79} (1980), 556--558. 

\bibitem[BD88]{BedDem} E. Bedford, J.-P. Demailly, {\sl Two Counterexamples Concerning the Pluri-Complex Green Function in $\C^n$}, Indiana Univ. Math. J. {\bf 37} (1988), 865--867.

\bibitem[BT76]{BedTay} E. Bedford, B.A. Taylor, {\sl The Dirichlet problem for a complex Monge-Amp\`ere equation}, Invent. Math. {\bf 37} (1976), 1--44.

\bibitem[Bło00]{Blocki} Z. Błocki, {\sl The $C^{1,1}$ regularity of the pluricomplex Green function}, Michigan Math. J. {\bf 47} (2000), 211-215.

\bibitem[BS92]{BS}  H.P. Boas, E.J. Straube, {\sl On equality of line type and variety type of real hypersurfaces in $\C^n$}, J. Geom. Anal. {\bf 2} (1992), 95--98. 

\bibitem[BCD10]{BCD} F. Bracci, M. D. Contreras, S. Diaz-Madrigal, {\sl Pluripotential theory, semigroups and boundary behavior of infinitesimal generators in strongly convex domains}, J. Eur. Math. Soc., {\bf 12}  (2010), no. 1, 23-53. 

\bibitem[BCD20]{booksemigroup} F. Bracci, M. D. Contreras, S. Diaz-Madrigal, {\sl Continuous Semigroups of holomorphic self-maps of the unit disc}, Springer Monographs in Mathematics, 566 pp., (2020).

\bibitem[BKZ21]{BKZ} F. Bracci, Ł. Kosinski, W. Zwonek, {\sl Slice rigidity property of holomorphic maps Kobayashi-isometrically preserving complex geodesics}, J. Geom. Anal.  {\bf 31} (2021), no. 11, 11292–11311.

\bibitem[BNT22]{BNT} F. Bracci, N. Nikolov, P. J. Thomas, {\sl Visibility of Kobayashi geodesics in convex domains and related properties}, Math. Z. {\bf 301} (2022), no. 2, 2011--2035.

\bibitem[BP05]{BP} F. Bracci, G. Patrizio, {\sl Monge-Ampère foliations with singularities at the boundary of strongly convex domains}, Math. Ann. {\bf 332} (2005), no. 3, 499--522.

\bibitem[BFW19]{BFW} F. Bracci, J. E. Fornaess, E. F. Wold, {\sl Comparison of invariant metrics and distances on strongly pseudoconvex domains and worm domains}, Math. Z. {\bf 292} (2019), no. 3-4, 879-893.
		
\bibitem[BPT09]{BPT}  F. Bracci, G. Patrizio and S. Trapani, {\sl The pluricomplex Poisson kernel for strongly convex domains}, Trans. Amer. Math. Soc.  {\bf 361} (2009), no. 2, 979--1005.

\bibitem[BS09]{Bracci-Saracco} F. Bracci, A. Saracco, {\sl Hyperbolicity in unbounded convex domains}, Forum. Math. {\bf 21} (2009), no. 5, 815--826.

\bibitem[BST21]{BST} F. Bracci, A. Saracco, S. Trapani, {\sl The pluricomplex Poisson kernel for strongly pseudoconvex domains}, Adv. Math. {\bf 380} (2021), 107577.

\bibitem[BT07]{BracciTrapani} F. Bracci, S. Trapani, {\sl Notes on pluripotential theory}, Rend. Mat. Appl. {\bf 27} (2007), 197–264.

\bibitem[BH99]{BH} M. Bridson, A. Haefliger, {\sl Metric spaces of nonpositive curvature}, Grundlehren der Mathematischen Wissenschaften [Fundamental Principles of Mathematical Sciences] {\bf 319}, Springer-Verlag, Berlin (1999).


\bibitem[CHL88]{CHL} C. H. Chang, M. C. Hu and H. P. Lee, {\sl Extremal analytic discs with prescribed boundary data}, Trans. Amer. Math. Soc. {\bf 310} (1988), 355--369.

\bibitem[deF03]{defabr} de Fabritiis, {\sl Complex geometry of generalized annuli},  Advances in Geometry {\bf 3} (2003), no. 4, 433--452.

\bibitem[Dem87]{Demailly} J.-P. Demailly, {\sl Mesures de Monge-Ampère et mesures pluriharmoniques}, Math. Z. {\bf194} (1987), 519--564.




\bibitem[Har79]{Harris} L. Harris, {\sl  Schwarz-Pick systems of pseudometrics for domains in normed linear spaces}, Advances in holomorphy, Notas de Matematica 65, North-Holland, Amsterdam (1979), pp. 345--406.

\bibitem[Hua95]{Huang} X. Huang {\sl A Boundary Rigidity Problem For Holomorphic Mappings on Some Weakly Pseudoconvex Domains}, Canad. J. Math. {\bf 47} (1995), 405--420.

\bibitem[HW22]{HuWa} X. Huang, X. Wang, {\sl Complex geodesics and complex Monge–Ampère equations with boundary singularity}, Math. Ann. \textbf{382}, 1825–1864 (2022).

\bibitem[JP13]{JP} M. Jarnicki, P. Pflug, {\sl Invariant Distances and Metrics in Complex Analysis}, Berlin, Boston: De Gruyter, (2013).

\bibitem[JPZ93]{JPZ} M. Jarnicki, P. Pflug, R. Zeinstra, {\sl Geodesics for convex complex ellipsoids}, Ann. Scuola Norm. Sup. Pisa, Serie 4,  {\bf 20} (1993), no. 4, 535--543.

\bibitem[Kli85]{Kli85}M. Klimek, {\sl Extremal plurisubharmonic functions and invariant pseudodistances}, Bull. Soc. Math. Fr. {\bf 113} (1985), 231--240.

\bibitem[Kli91]{KlimBook} M. Klimek, {\sl Pluripotential Theory. London Math. Soc. Monographs}, New Series, Academic Press, (1991).

\bibitem[Kob98]{kobayascione} S. Kobayashi, {\sl Hyperbolic Complex Spaces}, Grundlehren der Mathematischen Wissenschaften [Fundamental Principles of Mathematical Sciences] {\bf 318}, Springer-Verlag, Berlin (1998).


\bibitem[Lem81]{Lemp1} L. Lempert, {\sl La métrique de Kobayashi et la representation des domaines sur la boule}, Bull. Soc. Math. Fr. {\bf 109} (1981), 427--474.

\bibitem[Lem82]{Lem82}  L. Lempert, {\sl Holomorphic retracts and intrinsic metrics in convex domains}. Anal. Math. {\bf 8} (1982), 257--261.

\bibitem[Lem83]{Lemp2} L. Lempert, {\sl Solving the degenerate Monge-Ampère equation with one concentrated singularity}, Math. Ann. {\bf 263}, (1983), 515-532.

\bibitem[Mer93]{Mercer} P. R. Mercer, {\sl  Complex geodesics and iterates of holomorphic maps on convex domains in $\C^n$}, Trans.  Amer. Math. Soc. {\bf 338} (1993), no. 1, 201--211. 

\bibitem[Ran95]{Ran} T. Ransford, {\sl Potential Theory in the Complex Plane}, Cambridge Univ. Press, (1995).

\bibitem[RW83]{RW83} H. L. Royden, P.-M. Wong, {\sl Carathéodory and Kobayashi metrics on convex domains}, preprint (1983).

\bibitem[Sad81]{Sad} A. Sadullaev, {\sl Plurisubharmonic measures and capacities on complex manifolds}, Russian Math. Surveys, {\bf 36} (1981), 61--119.

\bibitem[McN92]{mcneal} J.D. McNeal, {\sl Convex domains of finite type}, J. Funct. Anal. {\bf 108}(2) (1992),  361--373.

\bibitem[N\"O24]{NO} N. Nikolov, A. Y. \"Okten, {\sl Strong localizations of the Kobayashi distance}, Proc. Amer. Math. Soc. {\bf 152} (2024), no. 6, 2439–2448.

\bibitem[Zim16]{Zim} A. Zimmer, {\sl Gromov hyperbolicity and the Kobayashi metric on convex domains of finite type}, Math. Ann. {\bf 365} (2016), no. 3-4, 1425--1498.
\end{thebibliography}
\end{document}